\documentclass[a4paper,11pt]{amsart}
\usepackage[T1]{fontenc}
\usepackage[english]{babel}
\usepackage{amsfonts, amsmath, amssymb, mathrsfs}
\usepackage{amsthm}
\usepackage{dsfont}
\usepackage{braket}
\usepackage{color}
\definecolor{darkbrown}{rgb}{.5,.2,0}

\usepackage{graphics}
\usepackage[colorlinks=true,urlcolor=darkbrown,citecolor=black,linkcolor=black]
{hyperref}
\def\doi[#1]{\href{https://doi.org/#1}{\texttt{DOI:\,#1}}}
\def\url[#1]{\href{#1}{\texttt{URL:\,#1}}}
\usepackage{soul} % to cross out text
\usepackage{enumitem}
\setlist[enumerate]{%
topsep=.5ex plus0.5ex minus0.2ex,  %additional space above the list
parsep=.5ex plus0.3ex minus0.1ex,  %space between paragraphs of an item: parsep
itemsep=.5ex plus0.5ex minus0.2ex, %space between items: itemsep+parsep
leftmargin=4ex,    %space to the left of the text in second, third, etc. lines
itemindent=0ex,    %indentation of first line against second line
labelwidth=2ex,    %width of label box, * is standard width depending on font
labelsep=1ex,      %space between label box and text in first line
listparindent=0ex, %indentation of first line of paragraph within an item
rightmargin=0ex,   %space to the right of the text
align=left,        %label alignment, default is right
font=\rmfamily,
label=\arabic*)}
\theoremstyle{plain}
\newtheorem{Cor}{Corollary}
\newtheorem{Lem}{Lemma}
\newtheorem{Pro}{Proposition}
\newtheorem{Thm}{Theorem}
\theoremstyle{definition}
\newtheorem{Exa}{Example}
\newtheorem{Que}{Question}
\newtheorem{Rem}{Remark}
\newtheorem*{Cor*}{Corollary}
\newtheorem*{Que*}{Question}
\newcommand{\C}{\mathbb{C}}
\newcommand{\N}{\mathbb{N}}
\newcommand{\R}{\mathbb{R}}
\newcommand{\cF}{\mathcal{F}}
\newcommand{\cH}{\mathcal{H}}
\newcommand{\cI}{\mathcal{I}}
\newcommand{\cK}{\mathcal{K}}
\newcommand{\fB}{\mathfrak{B}}
\newcommand{\fP}{\mathfrak{P}}
\newcommand{\fS}{\mathfrak{S}}
\newcommand{\fT}{\mathfrak{T}}
\newcommand{\ii}{\operatorname{i}}
\newcommand{\id}{\mathds{1}}
\DeclareMathOperator\aff{aff}
\DeclareMathOperator\conv{conv}
\DeclareMathOperator\ext{ext}
\DeclareMathOperator\lin{lin}
\renewcommand\P{\mathrm{P}}
\DeclareMathOperator\ri{ri}
\DeclareMathOperator\supp{supp}
\DeclareMathOperator\Tr{Tr}
\begin{document}
\selectlanguage{english}
\title[Faces, constraints, quantum theory]
{The face generated by a point, generalized affine constraints, and quantum theory}
%{The face generated by a point, generalized affine constraints, 
%and applications to quantum theory}
%
\author{Stephan Weis}
\author{Maksim Shirokov}
\begin{abstract}
We analyze faces generated by points in an arbitrary convex set and their relative 
algebraic interiors, which are nonempty as we shall prove. We show that by 
intersecting a convex set with a sublevel or level set of a generalized affine 
functional, the dimension of the face generated by a point may decrease by at most 
one. We apply the results to the set of quantum states on a separable Hilbert space. 
Among others, we show that every state having finite expected values of 
any two (not necessarily bounded) positive operators admits a decomposition into pure 
states with the same expected values. We discuss applications in quantum information 
theory.
\end{abstract}
\date{June 4th, 2020}
%
% 52 Convex and discrete geometry
%
% 52Axx		General convexity
%
% 47 Operator theory
%
% 47Axx		General theory of linear operators
%
% 81 Quantum theory
%
% 81Qxx		General mathematical topics and methods in quantum theory
%
\maketitle
%
%%%%%%%%%%%%%%%%%%%%%%%%%%%%%%%%%%%%%%%%%%%%%%%%%%%%%%%%%%%%%%%%%%%%%%%%%%%%
%%%%%%%%%%%%%%%%%%%%%%%%%%%%%%%%%%%%%%%%%%%%%%%%%%%%%%%%%%%%%%%%%%%%%%%%%%%%
%%%%%%%%%%%%%%%%%%%%%%%%%%%%%%%%%%%%%%%%%%%%%%%%%%%%%%%%%%%%%%%%%%%%%%%%%%%%
%%%%%%%%%%%%%%%%%%%%%%%%%%%%%%%%%%%%%%%%%%%%%%%%%%%%%%%%%%%%%%%%%%%%%%%%%%%%
%%%%%%%%%%%%%%%%%%%%%%%%%%%%%%%%%%%%%%%%%%%%%%%%%%%%%%%%%%%%%%%%%%%%%%%%%%%%
%
\section{Introduction}
Many tasks of mathematical physics and quantum communication theory require the
analysis of the convex geometry of intersections of convex sets with a 
well-known geometry and sublevel sets of one or more  generalized affine maps 
that take values in $\R\cup\{+\infty\}$. Our motivating example is the set of 
density operators on a separable Hilbert space with bounded expected values of 
one or more positive, generally, unbounded linear operators. The analysis of 
several important characteristics of quantum systems and channels leads to the
optimization over sets of density operators of the above type, see the monographs 
\cite{H-SCI,Wat,Wilde} and the research papers
\cite{B&D,
Giovannetti-etal2014,
HolevoShirokov2006,
Man,
WildeQi2018,
W-EBN}.
Therefore, our mission is to understand the convex geometry of these sets and to
enable the use of analytic techniques.
\par
We start with basics in Section~\ref{sec:generated-face}. Relying on the
{\em Kuratowski-Zorn lemma}, we show that the face generated by a point in a
convex set has a nonempty relative algebraic interior. We discuss corollaries
and examples and we describe the face generated by a point in the intersection
of two convex sets.
\par
In Section~\ref{sec:extreme-points} we show that by intersecting a convex set 
with a sublevel set or a level set of a generalized affine map, the dimension 
of the face generated by a point may decrease by at most one. This allows us 
to exploit gaps in the list of dimensions of faces. For example, if the convex 
set has no faces of dimension $1,2,\ldots,n$, then every extreme point of the 
intersection of the convex set with the sublevel or level sets of up to $n$ 
generalized affine maps is an extreme point of the original convex set.
\par
Beginning with Section~\ref{sec:extreme-points-quantum}, we study the class of 
generalized affine maps $\fS(\cH)\to[0,+\infty]$ on the set $\fS(\cH)$ of quantum 
states, defined as the expected value functionals $f_H:\rho\mapsto\Tr H\rho$ of 
positive (not necessarily bounded) operators $H$ on a separable Hilbert space 
$\cH$. The list of dimensions of faces of $\fS(\cH)$ has a gap between zero (pure 
state) and three (Bloch ball). Hence, every extreme point of the intersection of 
the sublevel or level sets of the expected value functionals $f_{H_1},f_{H_2}$ of 
two positive operators $H_1,H_2$ is a pure state. We also show that this is not 
true for more than two operators nor for classical states.
\par
In Section 5 we combine convex geometry with topology and measure theory. As the 
sublevel sets of $f_{H_1}$ and $f_{H_2}$
are closed, $\mu$-compact, and convex sets \cite{HolevoShirokov2006,P&Sh}, we are 
able to write each state in their intersection as the barycenter of a probability 
measure supported on the set of pure states in the intersection of the same 
sublevel sets. Although the level sets of $f_{H_1}$ and $f_{H_2}$ are not closed, 
every state having finite expected values regarding $H_1$ and $H_2$ admits a
decomposition into pure states that have the same expected values almost surely. 
As an example, any bipartite state with finite marginal energies can be decomposed 
into pure states with the same marginal energies.
\par
The results allow us to show that the supremum of any convex function on the 
intersection of the sublevel or level sets of $f_{H_1}$ and $f_{H_2}$ can be 
taken only over pure states, provided that this function is lower semicontinuous 
or upper semicontinuous and upper bounded. This result (in case $H_2=H_1$)
simplifies essentially definitions of several characteristics used in quantum 
information theory and adjacent fields of mathematical physics. These 
applications are considered in Section~\ref{sec:applications}.
\par
%
%%%%%%%%%%%%%%%%%%%%%%%%%%%%%%%%%%%%%%%%%%%%%%%%%%%%%%%%%%%%%%%%%%%%%%%%%%%%
%%%%%%%%%%%%%%%%%%%%%%%%%%%%%%%%%%%%%%%%%%%%%%%%%%%%%%%%%%%%%%%%%%%%%%%%%%%%
%%%%%%%%%%%%%%%%%%%%%%%%%%%%%%%%%%%%%%%%%%%%%%%%%%%%%%%%%%%%%%%%%%%%%%%%%%%%
%%%%%%%%%%%%%%%%%%%%%%%%%%%%%%%%%%%%%%%%%%%%%%%%%%%%%%%%%%%%%%%%%%%%%%%%%%%%
%%%%%%%%%%%%%%%%%%%%%%%%%%%%%%%%%%%%%%%%%%%%%%%%%%%%%%%%%%%%%%%%%%%%%%%%%%%%
%
\section{On the Face Generated by a Point}
\label{sec:generated-face}
We explore the face of a convex set generated by a point and the relative algebraic 
interior of such a face. The reader may recognize the finite-dimensional 
counterparts to our findings, for example from \cite{Rockafellar1970}.
\par
We work in the setting of a real vector space $V$ and a convex subset
$K\subseteq V$.
A subset $E\subseteq K$ is an {\em extreme set} (or an {\em extreme subset of $K$}
if we wish to emphasize the set $K$) if whenever $x\in E$ and
\[
x=(1-\lambda)y+\lambda z
\]
for some $\lambda\in(0,1)$ and $y,z\in K$, then $y$ and $z$ are also in $E$, see 
\cite{OBrien1976} for this definition. A point $x\in K$ is called an 
{\em extreme point} if $\{x\}$ is an extreme set. We denote the set of extreme points 
of $K$ by $\ext(K)$. A {\em face} of $K$ is a convex, extreme subset of $K$. 
Note that if $x$ is an extreme point, then $\{x\}$ is a face. As the intersection 
of an arbitrary family of faces of $K$ is a face of $K$, the smallest face $F_K(x)$ 
of $K$ that contains $x\in K$ exists. We call $F_K(x)$ 
{\em the face of $K$ generated by $x$}.
\par
A {\em linear combination} of $n\in\N$ points $x_1,\ldots,x_n\in V$ is a sum 
\[
\alpha_1x_1+\ldots+\alpha_nx_n
\] 
with weights $\alpha_i\in\R$, $i=1,\ldots,n$. The linear combination is an 
{\em affine combination} if $\alpha_1+\cdots+\alpha_n=1$ and a {\em convex combination} 
if $\alpha_1+\cdots+\alpha_n=1$ and $\alpha_i\geq 0$ for all $i=1,\ldots,n$. Given a 
subset $X\subseteq V$, the set $\aff(X)$ of all affine combinations of points from $X$ 
is the {\em affine hull} of $X$. This is the smallest affine subspace of $V$ containing 
$X$. The {\em translation vector space} $\lin(X)$ is the set of differences between 
each two points from $\aff(X)$.
\par
The {\em algebraic interior} of the convex set $K$ is the set of all points $x\in K$ 
such that for every straight line $g\subseteq V$ passing through $x$ the point $x$ 
lies in the interior of the intersection $K\cap g$. 
We call {\em relative algebraic interior} of $K$ the set $\ri(K)$ of all points 
$x\in K$ such that for every straight line $g\subseteq\aff(K)$ passing through $x$ 
the point $x$ lies in the interior of the intersection $K\cap g$. 
\par
\begin{Lem}\label{lem:relint}
Let $C\subseteq K$ be a convex subset, $E\subseteq K$ an extreme subset,
$F\subseteq K$ a face of $K$, and let $x\in K$ be a point. Then
\begin{enumerate}
\item
$\ri(C)\cap E\neq\emptyset\implies C\subseteq E$,
\item
$x\in F\iff F_K(x)\subseteq F$,
\item
$x\in\ri(F)\implies F=F_K(x)$.
\end{enumerate}
\end{Lem}
\begin{proof}
1) Let $x\in\ri(C)$ and $y\in C$ with $y\neq x$. Then $x$ is an interior point of
the intersection $C\cap g$ of $C$ with the line $g$ through $x$ and $y$. Hence,
if $x$ lies in the extreme set $E$ so does $y$.
2) The inclusion $F_K(x)\subseteq F$ is true as $F_K(x)$ is the minimal face
containing $x$. The converse is obvious as $x\in F_K(x)$.
3) The inclusion $F\subseteq F_K(x)$ follows from part 1) with $C=F$ and $E=F_K(x)$.
The inclusion $F_K(x)\subseteq F$ follows from part 2).
\end{proof}
Corollary~\ref{cor:char-ri} provides the converse to Lemma~\ref{lem:relint}, part 3).
\par
\begin{Lem}\label{lem:ri-convex-open}
The complement $K\setminus\ri(K)$ of the relative algebraic interior $\ri(K)$ is an 
extreme subset of $K$ and $\ri(K)$ is a convex set. 
\end{Lem}
\begin{proof}
By the definition of the relative algebraic interior we have
\[
K\setminus\ri(K)=
\{x\in K \mid \exists v\in\lin(K): x+\epsilon v\not\in K\; \forall \epsilon>0\}.
\]
The right-hand side is an extreme set. Indeed, let $x,y,z\in K$,
$\lambda\in(0,1)$, and $x=(1-\lambda)y+\lambda z$. If $y\in\ri(K)$, then for all
vectors $v\in\lin(K)$ there is $\epsilon_{y,v}>0$ such that
$y+\epsilon_{y,v}v\in K$. Hence we have $x+(1-\lambda)\epsilon_{y,v} v\in K$, that 
is to say $x\in\ri(K)$. Similarly, $x\in K\setminus\ri(K)$ implies 
$z\in K\setminus\ri(K)$, which proves that $K\setminus\ri(K)$ is an extreme set. 
\par
The set $\ri(K)$ is convex as it is the complement of an extreme set. Indeed, it is 
easy to show that the complement $K\setminus S$ of a subset $S\subseteq K$ is convex 
if and only if $(1-\lambda)y+\lambda z\in S$ implies that at least one of the points 
$y$ {\em or} $z$ lies in $S$ for all $y,z\in K$ and $\lambda\in(0,1)$, while both
$y$ {\em and} $z$ needed to lie in $S$ if $S$ were an extreme set.
\end{proof}
It is well known that nonempty convex sets may have empty relative algebraic
interiors, see Section~III.1.6  of \cite{Barvinok2002} and the
Examples~\ref{exa:unipol} and~\ref{exa:discrete-pms} below. This is not the case 
for faces generated by points in a convex set.
\par
\begin{Thm}\label{thm:aff}
Let $x$ be a point in $K$. The affine hull of the face of $K$ generated by $x$ is
\begin{equation}\label{eq:aff}
\aff F_K(x)=\{y\in V \mid \exists\epsilon>0 : x\pm\epsilon(y-x)\in F_K(x)\}.
\end{equation}
In particular, $x$ lies in the relative algebraic interior of $F_K(x)$.
\end{Thm}
\begin{proof}
The second assertion follows from equation \eqref{eq:aff} and from the definition 
of the relative algebraic interior. We prove the equation \eqref{eq:aff}. As the 
inclusion ``$\supseteq$'' is clear, it suffices to prove ``$\subseteq$''. First, 
note that for all $v\in V$ the set
\[
E_v=\{y\in F_K(x) \mid y+\epsilon v\not\in F_K(x)\; \forall \epsilon>0\}
\]
is an extreme subset of $F_K(x)$. The proof is similar to the proof of
Lemma~\ref{lem:ri-convex-open}.
\par
The main idea is as follows. If $x+v\in\aff(F_K(x))$, then it
follows that $E_v$ and $E_{-v}$ are proper subsets of $F_K(x)$, as we detail
below. If $x$ lies in $E_v$, then according to the Kuratowski-Zorn
lemma, there is a maximal convex subset $C$ of $E_v$ containing $x$. Below,
we show that $C$ is a face of $F_K(x)$. Since $F_K(x)$ is the minimal face
containing $x$, this implies that $x$ lies outside of $E_v$. Similarly, $x$
lies outside of $E_{-v}$. This means that there are $\epsilon_1,\epsilon_2>0$
such that $x+\epsilon_1 v,x-\epsilon_2 v\in F_K(x)$. As $F_K(x)$ is convex,
it follows that $x\pm\epsilon v\in F_K(x)$ where
$\epsilon=\min(\epsilon_1,\epsilon_2)$. In other words, $x+v$ lies in the
right-hand side of \eqref{eq:aff}.
\par
As promised, we show that $E_v$ and $E_{-v}$ are proper subsets of $F_K(x)$
if $x+v\in\aff(F_K(x))$. Let $x+v=\sum_{i=1}^n\alpha_ix_i$ and
$x-v=\sum_{j=1}^m\beta_jy_j$ be affine combinations of points
$x_1,\ldots,x_n,y_1,\ldots,y_m$ from $F_K(x)$ and define
$M=(|\alpha_1|+\cdots+|\alpha_n|+|\beta_1|+\cdots+|\beta_m|)/2$. The points
\[
y=
\sum_{\substack{i=1\\\alpha_i> 0}}^n\tfrac{\alpha_i}{M}x_i
-\sum_{\substack{j=1\\\beta_j< 0}}^m\tfrac{\beta_j}{M}y_j
\quad\text{and}\quad
z=
\sum_{\substack{j=1\\\beta_j> 0}}^m\tfrac{\beta_j}{M}y_j
-\sum_{\substack{i=1\\\alpha_i< 0}}^n\tfrac{\alpha_i}{M}x_i
\]
lie in $F_K(x)$ and $v=\tfrac{M}{2}(y-z)$ holds. Then $y=z+\tfrac{2}{M}v$ shows
that $z\not\in E_v$, and $z=y-\tfrac{2}{M}v$ shows that $y\not\in E_{-v}$.
\par
\begin{figure}[h]
\includegraphics{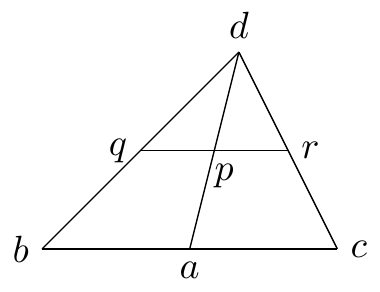}
\caption{Sketch for the proof of Theorem~\ref{thm:aff}.}
\label{fig:face}
\end{figure}
To complete the proof, we have to show that every maximal convex subset $C$ of $E_v$
containing $x$ is a face of $F_K(x)$. As $C$ is convex, it suffices to show
that $C$ is an extreme subset of $F_K(x)$. Let $a\in C$ be arbitrary and let
$a=(1-\lambda)b+\lambda c$ where $b,c$ are from $F_K(x)$ and $\lambda\in(0,1)$.
Since $C$ is a maximal convex subset of $E_v$ containing $x$, the claim follows if
we show that $E_v$ contains the convex hull of $C\cup\{b\}$. Let $d\in C$ be
arbitrary and let $q=(1-\mu)b+\mu d$ where $\mu\in[0,1]$. Also, define
$p=(1-\mu)a+\mu d$ and $r=(1-\mu)c+\mu d$, see Figure~\ref{fig:face}. Then 
$p=(1-\lambda)q+\lambda r$. As $p\in[a,d]\subseteq C\subseteq E_v$, and as $E_v$ is an 
extreme set, it follows that $q\in E_v$.
\end{proof}
The extreme set $E_v$ in the proof of Theorem~\ref{thm:aff}
is not convex in general. As an example, consider the square $K=[-1,1]\times[-1,1]$,
$x=(0,0)$, and $v=(1,1)$. Then $F_K(x)=K$ and
$E_v=\{(\eta,\xi)\in K : \text{$\eta=1$ or $\xi=1$} \}$ is the nonconvex union
of two perpendicular segments. See also Corollary~\ref{cor:extreme}.
\par
Note that (because $F_K(x)$ is an extreme subset of $K$) equation \eqref{eq:aff}
implies that the affine hull of the face of $K$ generated by a point $x\in K$ is
\begin{equation}\label{eq:affK}
\aff F_K(x)=\{y\in V \mid \exists\epsilon>0 : x\pm\epsilon(y-x)\in K\}.
\end{equation}
\par
\begin{Cor}\label{cor:facex}
Let $x\in K$. Then $F_K(x)=\bigcup_{y,z\in K,x\in(y,z)}[y,z]$. The right-hand
side is the union over all closed segments in $K$ for which $x$ lies on the open
segment. (By definition $[x,x]=(x,x)=\{x\}$.)
\end{Cor}
\begin{proof}
The inclusion ``$\supseteq$'' follows as $F_K(x)$ is an extreme set containing
$x$. Conversely, let $y\in F_K(x)$. By equation \eqref{eq:affK} there
is $\epsilon>0$ such that the point $z_-=x-\epsilon(y-x)$ lies in $K$. Then
$x=\tfrac{1}{1+\epsilon}z_-+\tfrac{\epsilon}{1+\epsilon}y$ shows that $x$ lies
in the open segment $(z_-,y)$. This completes the proof.
\end{proof}
Corollary~\ref{cor:facex} shows that the closure of $F_K(x)$ in a topological
vector space would be the {\em face function} of $K$ at a point $x$ in $K$, as 
studied in \cite{KleeMartin1971}. A subset $E$ of $K$ satisfying the property 2) 
of Corollary~\ref{cor:extreme} below is called an {\em extreme set} in the paper
\cite{Papadopoulou1977}. The faces with nonempty relative algebraic interiors 
are the building blocks of extreme sets and of convex sets in the sense of 
Corollary~\ref{cor:extreme}, part 3), and Corollary~\ref{cor:partition},
respectively.
\par
\begin{Cor}\label{cor:extreme}
Let $E\subseteq K$ be a subset. The following assertions are equivalent.
The set $E$
\begin{enumerate}
\item
is an extreme subset of $K$,
\item
contains the face $F_K(x)$ of $K$ generated by any point $x$ in $E$,
\item
is a union of faces of $K$ having nonempty relative algebraic interiors,
\item
is a union of faces of $K$.
\end{enumerate}
\end{Cor}
\begin{proof}
The implication 1) $\Rightarrow$ 2) follows from Corollary~\ref{cor:facex}. The 
implication 2) $\Rightarrow$ 3) follows from Theorem~\ref{thm:aff} as $x$ is 
a relative algebraic interior point of $F_K(x)$ for all $x\in K$. 
The implication 3) $\Rightarrow$ 4) is obvious and 4) $\Rightarrow$ 1) is true 
as every union of extreme sets is an extreme set.
\end{proof}
\begin{Cor}\label{cor:partition}
The family of relative algebraic interiors of faces of $K$ is a partition of $K$.
\end{Cor}
\begin{proof}
The family $\{\ri(F_K(x)):x\in K\}$ covers $K$, as Theorem~\ref{thm:aff} shows
$x\in\ri(K)$ for all $x\in K$. Let $F,G$ be faces of $K$ that intersect in their
relative algebraic interiors, say $x\in\ri(F)\cap\ri(G)$. Then part 3) of
Lemma~\ref{lem:relint} provides $F=G=F_K(x)$. This proves the claim.
\end{proof}
We characterize relative algebraic interiors of faces.
\par
\begin{Cor}\label{cor:char-ri}
Let $F$ be a face of $K$ and $x$ a point in $K$. Then 
\[
x\in\ri(F) \iff F=F_K(x).
\]
In particular, if $x$ and $y$ are points in $K$, then 
\[
x\in\ri(F_K(y)) \iff F_K(y)=F_K(x).
\]
\end{Cor}
\begin{proof}
The first statement follows directly from part 3) of Lemma~\ref{lem:relint}
and Theorem~\ref{thm:aff}. The second statement is the special case
$F=F_K(y)$ of the first statement.
\end{proof}
\begin{Exa}[Univariate polynomials]\label{exa:unipol}
Let $V$ be the vector space of all countably infinite sequences of real numbers such
that all but finitely many terms are zero. Each nonzero element from $V$ may be
written in the form
\[
v=(a_0,a_1,a_2,\ldots,a_{n-1},a_n,0,0,0,\ldots),
\qquad a_n\neq 0
\]
where $a_i\in\R$ for all $i=0,\ldots,n$ and $n\in\N_0=\{0,1,2,\ldots\}$. There is a 
one-to-one correspondence between $V$ and the space $\R[x]$ of all polynomials in one 
variable $x$ with real coefficients, {\em via} the identification of the above vector 
$v$ with the nonzero polynomial
\begin{equation}\label{eq:poly}
p=a_nx^n+a_{n-1}x^{n-1}+\cdots+a_2x^2+a_1x+a_0,
\qquad a_n\neq 0.
\end{equation}
The {\em degree} of $p$ is $n=\deg(p)$ and the {\em leading coefficient} is $a_n$.
\par
It is well known that the convex set $K\subset\R[x]$ of polynomials with positive 
leading coefficients has an empty relative algebraic interior \cite{Barvinok2002}. 
Below in Lemma~\ref{lem:polynomials} we show that the subsets of $K$ consisting of
polynomials of constant degrees are relative algebraic interiors of faces of $K$. 
As the family of relative algebraic interiors of all faces is a partition of $K$ 
(see Corollary~\ref{cor:partition}), we corroborate that the relative algebraic 
interior $\ri(K)$ is empty.
\end{Exa}
\begin{Lem}\label{lem:polynomials}
Let $K\subset\R[x]$ denote the convex set of univariate polynomials with positive
leading coefficients. The face of $K$ generated by a polynomial $p\in K$ is 
$F_K(p)=F_{\deg(p)}$, where 
\[
F_n=\{q\in K \mid \deg(q)\leq n\},
\qquad n\in\N_0.
\]
The relative algebraic interior of $F_n$ is 
\[
\ri(F_n)=\{q\in K \mid \deg(q)=n\},
\qquad n\in\N_0,
\]
which is an open half-space of dimension $n+1$. Every nonempty extreme subset of $K$
is equal to $K$ or to one of the faces $F_n$, $n\in\N_0$.
\end{Lem}
\begin{proof}
Let $p,q\in K$ be polynomials with positive leading coefficients. 
Corollary~\ref{cor:facex} shows that $q$ lies in $F_K(p)$ if and only if there is 
$\lambda<0$ such that $(1-\lambda)p+\lambda q\in K$, which is equivalent to 
$\deg(q)\leq\deg(p)$. This proves  that $F_K(p)=F_{\deg(p)}$. 
Corollary~\ref{cor:char-ri} shows that $q$ lies in $\ri\left(F_K(p)\right)$ if and 
only if $F_K(p)=F_K(q)$, or equivalently $\deg(p)=\deg(q)$.
\par
We show that every nonempty extreme subset $E$ of $K$ equals $K$ or $F_n$ for some $n\in\N_0$. Let $I=\{\deg(p) : p\in E\}$. Then $E=\bigcup_{i\in I}F_n$ holds as 
$p\in E$ implies $F_{\deg(p)}=F_K(p)\subseteq E$ by Corollary~\ref{cor:extreme}. 
Since $F_0\subseteq F_1\subseteq F_2\subseteq\cdots$, this implies $E=F_n$ if 
$\sup(I)=n<\infty$ and $E=K$ if $\sup(I)=\infty$.
\end{proof}
\begin{Exa}[Discrete probability measures]\label{exa:discrete-pms}
The set of probability measures on the set of natural numbers $\N=\{1,2,3,\ldots\}$
is affinely isomorphic to the set 
\[
\Delta_\N=
\left\{p:\N\to\R\mid
\text{$p(n)\geq 0\;\forall n\in\N$, and $\sum_{n\in\N}p(n)=1$}\right\}
\]
of probability densities with respect to the counting measure. The set of extreme 
points $\ext(\Delta_\N)$ consists of the densities $\delta_n(m)=\left\{
\begin{smallmatrix}\text{$1$ if $m=n$}\\\text{$0$ if $m\neq n$}\end{smallmatrix}
\right.$, $m\in\N$, concentrated at the points $n\in\N$. We may think of 
$\Delta_\N$ as a {\em simplex}, or better a {\em $\sigma$-simplex}, as each density 
$p\in\Delta_\N$ can be written in a unique way as a countable convex combination 
$p=\sum_{n\in\N}\lambda_n\delta_n$ of extreme points, where $\lambda_n\geq 0$ for 
all $n\in\N$, and $\sum_{n\in\N}\lambda_n=1$.
\par
Let $\,\supp(p)=\{n\in\N:p(n)>0\}$ denote the {\em support} of a density 
$p\in\Delta_\N$ and 
\[
\Delta_I=
\left\{p\in\Delta_\N\mid\supp(p)\subseteq I\right\}
\]
the set of densities supported on a subset $I\subseteq\N$. The convex set $\Delta_I$ 
is an extreme subset and hence a face of $\Delta_\N$ as 
\begin{equation}\label{eq:supp-union}
\supp\left((1-\lambda)p+\lambda q\right)=\supp(p)\cup\supp(q)
\end{equation} 
holds for all $p,q\in\Delta_\N$ and $\lambda\in(0,1)$. Therefore, 
$\ext(\Delta_I)=\{\delta_n\mid n\in I\}$. The convex hull 
\[
\conv\left(\{\delta_n\mid n\in I\}\right)
=
\{p\in\Delta_I : |\supp(p)|<\infty\}
\]
is the set of densities with finite support in $I$. Again by equation 
\eqref{eq:supp-union}, the convex set $\,\conv\left(\{\delta_n\mid n\in I\}\right)$ 
is an extreme subset and hence a face of $\Delta_\N$. 
\par
For finite subsets $I\subset\N$ we have 
$\,\conv\left(\{\delta_n\mid n\in I\}\right)=\Delta_I$. Lemma~\ref{lem:discrete-pms} 
below shows $F_{\Delta_\N}(p)=\Delta_{\supp(p)}$, which has relative algebraic interior 
\[
\ri(F_{\Delta_\N}(p))=\{q\in\Delta_\N:\supp(q)=\supp(p)\}
\] 
for all densities $p\in\Delta_\N$ of finite support. 
\par
Let $I\subseteq\N$ be infinite. Then the relative algebraic interior of the face 
$\,\conv\left(\{\delta_n\mid n\in I\}\right)$ is empty. This follows from part 3) of 
Lemma~\ref{lem:relint}, as $F_{\Delta_\N}(p)=\Delta_{\supp(p)}$ is strictly included 
in $\,\conv\left(\{\delta_n\mid n\in I\}\right)$ for all densities $p$ with finite 
support in $I$. The relative algebraic interior of the simplex $\Delta_I$ is empty, 
too, but for different reasons aside from the support sizes. Let us consider the 
interval
\[
\cI(I)=
\left\{\text{$F$ is a face of $\Delta_\N$} \mid
\conv\left(\{\delta_n\mid n\in I\}\right)
\subseteq F\subseteq\Delta_I \right\},
\]
partially ordered by inclusion. Lemma~\ref{lem:discrete-pms} shows that the face
$F_{\Delta_\N}(p)$ belongs to $\cI(I)$ if and only if $\supp(p)=I$ and that the
inclusion of such faces is determined by the asymptotics of converging series.
Following an example by Hadamard \cite{Hadamard1894}, we define the map $p_H:\N\to\R$,
\[
p_H(n)
=p(n)/(\sqrt{r_n}+\sqrt{r_{n+1}})
=\sqrt{r_n}-\sqrt{r_{n+1}},
\qquad n\in\N,
\]
for every density $p\in\Delta_\N$ with support $I$ and $r_n=\sum_{m\geq n}p(m)$. The 
map $p_H$ is a probability density with support $I$, and for $n\in I$ we have
\[
p(n)/p_H(n)
=\sqrt{r_n}+\sqrt{r_{n+1}}
\stackrel{n\to\infty}{\longrightarrow}0
\quad\text{and}\quad
p_H(n)/p(n)\stackrel{n\to\infty}{\longrightarrow}\infty.
\]
Lemma~\ref{lem:discrete-pms} shows that $p\in F_{\Delta_\N}(p_H)$ and
$p_H\not\in F_{\Delta_\N}(p)$. Part 2) of Lemma~\ref{lem:relint} then implies that
$\cI(I)$ contains the infinite chain of strictly included faces
\begin{equation}\label{eq:Hadamard-chain}
F_{\Delta_\N}(p)
\subset F_{\Delta_\N}(p_H)
\subset F_{\Delta_\N}((p_H)_H)
\subset\cdots
\end{equation}
The strict inclusions $F_{\Delta_\N}(p)\subset\Delta_I$ for all densities $p$ with
support $I$, the strict inclusions $F_{\Delta_\N}(q)\subseteq\Delta_J\subset\Delta_I$
for all densities $q$ with support $J\subset I$, and part 3) of Lemma~\ref{lem:relint}
show $\ri(\Delta_I)=\emptyset$. 
\par
We close the example with a glimpse at the interval $\cI(\N)$. We consider
the Euler-Riemann zeta function $\zeta(s)=\sum_{n\in\N}n^{-s}$ and the map 
$p_s:\N\to\R$, $n\mapsto\zeta(s)^{-1}\cdot n^{-s}$ for all
$s>1$. Lemma~\ref{lem:discrete-pms} shows that $p_s\in F_{\Delta_\N}(p_t)$ holds 
if and only if $t\leq s$ for all $s,t>1$. Hence, part 2) of Lemma~\ref{lem:relint} 
proves
\[
F_{\Delta_\N}(p_s) \subseteq F_{\Delta_\N}(p_t)
\iff
t\leq s,
\qquad s,t>1.
\]
The interval $\cI(\N)$ contains the uncountable chain of faces
$\{F_{\Delta_\N}(p_s) : s>1\}$.
\end{Exa}
\begin{Que}\label{que:faces-hyper}
We noted in Example~\ref{exa:discrete-pms} that the partial ordering of the faces 
generated by points of the $\sigma$-simplex $\Delta_\N$ 
is governed by the asymptotics of converging series, a classical topic of real 
analysis \cite[Section~41]{Knopp1990}. Could the $\sigma$-simplex $\Delta_\N$ 
provide a convex geometry approach to the theory of series? 
\par
We showed in Example~\ref{exa:discrete-pms} that the convex sets 
$\,\conv\left(\{\delta_n\mid n\in I\}\right)$ and $\Delta_I$ are faces of the 
$\sigma$-simplex $\Delta_\N$ and that they have empty relative algebraic interiors 
for all infinite subsets $I\subseteq\N$. Are there any other faces of $\Delta_\N$ 
that also have empty relative algebraic interiors? 
\end{Que}
\begin{Lem}\label{lem:discrete-pms}
The face of $\Delta_\N$ generated by a density $p\in\Delta_\N$ is
\begin{equation}\label{eq:faceDelta}
F_{\Delta_\N}(p)=
\{q\in\Delta_{\supp(p)} \mid \sup_{n\in\supp(p)}q(n)/p(n)<\infty\}.
\end{equation}
The face $F_{\Delta_\N}(p)$ has the relative algebraic interior
\begin{equation}\label{eq:ri-faceDelta}
\ri\left(F_{\Delta_\N}(p)\right)=
\{q\in F_{\Delta_\N}(p) \mid
\inf_{n\in\supp(p)}q(n)/p(n)>0\}.
\end{equation}
\end{Lem}
\begin{proof}
Corollary~\ref{cor:facex} shows that a density $q\in\Delta_\N$ lies in
$F_{\Delta_\N}(p)$ if and only if there is $\lambda<0$ such that
$(1-\lambda)p+\lambda q\in\Delta_\N$. If such a $\lambda<0$ exists, $q(n)=0$ holds 
for all $n\not\in\supp(p)$ and $\tfrac{q(n)}{p(n)}\leq\tfrac{1-\lambda}{-\lambda}$ 
holds for all $n\in\supp(p)$. Conversely, let $q\in\Delta_{\supp(p)}$ and let
$\mu=\supp_{n\in\supp(p)}\tfrac{q(n)}{p(n)}<\infty$. As $q$ is a probability
density supported on $\supp(p)$, we have $\mu\geq 1$ with equality if and
only if $p=q$. As $p\in F_{\Delta_\N}(p)$, we may assume $\mu>1$ and define 
$\lambda=\tfrac{-1}{\mu-1}$. Then
\[
\tfrac{\mu}{\mu-1}p(n)+\tfrac{-1}{\mu-1}q(n)
\geq
\tfrac{\mu}{\mu-1}p(n)+\tfrac{-1}{\mu-1}\mu p(n)
=0
\qquad
\forall n\in\supp(p)
\]
completes the proof of equation \eqref{eq:faceDelta}.
\par
Corollary~\ref{cor:char-ri} shows that a density function $q\in\Delta_\N$ lies in
$\ri\left(F_{\Delta_\N}(p)\right)$ if and only if $F_{\Delta_\N}(p)=F_{\Delta_\N}(q)$.
By equation \eqref{eq:faceDelta} this is equivalent to $\supp(p)=\supp(q)$ and
\begin{equation}\label{eq:asymptote}
\sup_{n\in\supp(p)}\frac{r(n)}{p(n)}<\infty
\quad\iff\quad
\sup_{n\in\supp(p)}\frac{r(n)}{q(n)}<\infty
\qquad
\forall r\in\Delta_{\supp(p)}.
\end{equation}
To prove the equivalence of \eqref{eq:asymptote} and the conditions specifying the 
right-hand side of \eqref{eq:ri-faceDelta}, it suffices to assume $\supp(p)=\supp(q)$ 
and to prove that \eqref{eq:asymptote} is equivalent to
\begin{equation}\label{eq:qp}
\sup_{n\in\supp(p)}q(n)/p(n)<\infty
\quad\text{and}\quad
\sup_{n\in\supp(p)}p(n)/q(n)<\infty.
\end{equation}
If the first condition of \eqref{eq:qp} fails, then $r=q$ shows that
\eqref{eq:asymptote} fails. Similarly, if the second condition fails then $r=p$
shows that \eqref{eq:asymptote} fails. Conversely, if \eqref{eq:qp} is true, then
\[
\sup_{n\in\supp(p)}\frac{r(n)}{p(n)}
=\sup_{n\in\supp(p)}\frac{r(n)}{q(n)}\cdot \frac{q(n)}{p(n)}
\leq\sup_{n\in\supp(p)}\frac{r(n)}{q(n)}
\cdot\sup_{n\in\supp(p)}\frac{q(n)}{p(n)}
\]
proves the implications ``$\Leftarrow$'' of \eqref{eq:asymptote}.
Similarly, we prove the opposite implications.
\end{proof}
The face generated by a point in the intersection of two convex sets is easily 
described in terms of the individual sets.
\par
\begin{Pro}\label{pro:intersection}
Let $K,L\subseteq V$ be two convex sets and let $x\in K\cap L$. Then
\begin{enumerate}
\item
\parbox{5cm}{\raggedleft $F_{K\cap L}(x)$}
$=$
\parbox{5cm}{$F_K(x)\cap F_L(x)$,}
\item
\parbox{5cm}{\raggedleft $\ri\big(F_K(x)\cap F_L(x)\big)$}
$=$
\parbox{5cm}{$\ri\big(F_K(x)\big)\cap\ri\big(F_L(x)\big)$,}
\item
\parbox{5cm}{\raggedleft $\aff\big(F_K(x)\cap F_L(x)\big)$}
$=$
\parbox{5cm}{$\aff\big(F_K(x)\big)\cap\aff\big(F_L(x)\big)$.}
\end{enumerate}
\end{Pro}
\begin{proof}
Part 3). We prove the equation by demonstrating the two inclusions
\begin{equation}\label{eq:affhull-ext}
\aff\big(F_{K\cap L}(x)\big)
\subseteq\aff\big(F_K(x)\cap F_L(x)\big)
\subseteq\aff\big(F_K(x)\big)\cap\aff\big(F_L(x)\big)
\end{equation}
and the equality of the first and third terms of \eqref{eq:affhull-ext}. The
latter follows from equation
\eqref{eq:affK}, according to which for all $y\in V$ we have
\begin{align*}
 &
y\in\aff\left(F_{K\cap L}(x)\right)\\
 \iff&
\exists\epsilon>0 : x\pm\epsilon(y-x)\in K\cap L\\
 \iff&
\left(\exists\epsilon_K>0 : x\pm\epsilon_K(y-x)\in K\right) \wedge
\left(\exists\epsilon_L>0 : x\pm\epsilon_L(y-x)\in L\right)\\
 \iff&
y\in\aff\left(F_K(x)\right)\cap\aff\left(F_L(x)\right).
\end{align*}
It is clear that $F=F_K(x)\cap F_L(x)$ is a face of $K\cap L$ containing $x$.
Hence $F_{K\cap L}(x)\subseteq F$, which implies the first inclusion
of~\eqref{eq:affhull-ext}. The second inclusion follows because $\aff(F)$
is the smallest affine space containing $F$.
\par
Part 2). The inclusion ``$\supseteq$'' follows from the definition of the relative
algebraic interior and from
$\aff(F)\subseteq\aff\left(F_K(x)\right)\cap\aff\left(F_L(x)\right)$ provided
in part 3). To prove the inclusion ``$\subseteq$'' it suffices to show
$\ri(F)\subseteq\ri\left(F_K(x)\right)$. Lemma~\ref{lem:ri-convex-open} shows that
$\ri\left(F_K(x)\right)$ is the complement of an extreme subset of $F_K(x)$. As
$F$ intersects $\ri\left(F_K(x)\right)$, namely in $x$, part 1) of
Lemma~\ref{lem:relint} shows that $\ri(F)\subseteq\ri\left(F_K(x)\right)$.
\par
Part 1). By Theorem~\ref{thm:aff}, the point $x$ lies in 
$\ri F_K(x)\cap \ri F_L(x)$. This and the inclusion ``$\supseteq$'' of part 2) imply 
that $x$ lies in the relative algebraic interior of the face $F$ of $K\cap L$. Thus, 
part 3) of Lemma~\ref{lem:relint} proves the claim.
\end{proof}
Note that the assertions of Proposition~\ref{pro:intersection} are simplified according to
the rules $F_L(x)=L=\ri(L)=\aff(L)$ if $L$ is an affine space (incident with $x$).
\par
%
%%%%%%%%%%%%%%%%%%%%%%%%%%%%%%%%%%%%%%%%%%%%%%%%%%%%%%%%%%%%%%%%%%%%%%%%%%%%
%%%%%%%%%%%%%%%%%%%%%%%%%%%%%%%%%%%%%%%%%%%%%%%%%%%%%%%%%%%%%%%%%%%%%%%%%%%%
%%%%%%%%%%%%%%%%%%%%%%%%%%%%%%%%%%%%%%%%%%%%%%%%%%%%%%%%%%%%%%%%%%%%%%%%%%%%
%%%%%%%%%%%%%%%%%%%%%%%%%%%%%%%%%%%%%%%%%%%%%%%%%%%%%%%%%%%%%%%%%%%%%%%%%%%%
%%%%%%%%%%%%%%%%%%%%%%%%%%%%%%%%%%%%%%%%%%%%%%%%%%%%%%%%%%%%%%%%%%%%%%%%%%%%
%
\section{Convex Sets under Generalized Affine Constraints}
\label{sec:extreme-points}
We study sublevel sets and level sets of generalized affine maps on convex sets
and we analyze the dimensions of the faces generated by their points. In 
addition to general convex sets, we discuss the class of pyramids.
\par
As before, let $V$ be a real vector space and $K\subseteq V$ a convex subset. 
A map $f:K\to\R$ is called an {\em affine map} if
\begin{equation}\label{eq:affine}
f(\lambda x+\mu y)=\lambda f(x)+\mu f(y),
\qquad
x,y\in K,
\quad
\lambda,\mu\geq 0,
\quad \lambda+\mu=1.
\end{equation}
Consider the extended real line $\R\cup\{+\infty\}$ with the obvious ordering and 
arithmetics\footnote{%
We have $\alpha<+\infty$ for all $\alpha\in\R$. The addition is defined as 
$(+\infty)+(+\infty)=+\infty$ and $\alpha+(+\infty)=(+\infty)+\alpha=+\infty$ 
for all $\alpha\in\R$. The multiplication with non-negative scalars is defined 
as $0\cdot(+\infty)=(+\infty)\cdot 0=0$ and
$\lambda\cdot(+\infty)=(+\infty)\cdot\lambda=+\infty$ for all $\lambda>0$.}.
We call 
\[
f:K\to\R\cup\{+\infty\}
\]
a {\em generalized affine map on $K$} if $f$ satisfies the equation (\ref{eq:affine})
with the extended arithmetics. Let $\alpha\in\R$ and let
\begin{align}\label{eq:simple-const}
K_f &=\{x\in K: f(x)<+\infty\},\nonumber\\
K_f^\leq &=\{x\in K: f(x)\leq\alpha\}, && \text{(sublevel set)}\\
K_f^= &=\{x\in K: f(x)=\alpha\}. && \text{(level set)}\nonumber
\end{align}
If $\alpha\in\R$ is unspecified, we assume the sublevel set $K_f^\leq$ and 
level set $K_f^=$ are taken at the same value of $\alpha$.
\begin{Lem}\label{lem:GenFinite}
The set $K_f$ is a face of $K$.
\end{Lem}
\begin{proof}
Let $x,y\in K$, $\lambda\in[0,1]$, and $z=(1-\lambda)x+\lambda y$ throughout the 
proof.
\par
To show that $K_f$ is convex, we assume that $x,y\in K_f$. Then $f(x)<+\infty$ and
$f(y)<+\infty$ gives
\[
f(z)
=f((1-\lambda)x+\lambda y)
=(1-\lambda) f(x)+\lambda f(y)
<+\infty,
\]
which means $z\in K_f$. To show that $K_f$ is an extreme set, it suffices to
show that $x\not\in K_f$ or $y\not\in K_f$ implies $z\not\in K_f$ for all
$\lambda\in(0,1)$. If $x\not\in K_f$, then $f(x)=+\infty$ yields
\[
f(z)
=f((1-\lambda)x+\lambda y)
=(1-\lambda) f(x)+\lambda f(y)
=+\infty+\lambda f(y)
=+\infty,
\]
which means $z\not\in K_f$. Similarly, $y\not\in K_f\implies z\not\in K_f$.
\end{proof}
\begin{Lem}\label{lem:sub-level}
The level set $K_f^=$ is a face of the sublevel set $K_f^\leq$.
\end{Lem}
\begin{proof}
Let $\alpha\in\R$, let $x,y,z\in K_f^\leq$, let $\lambda\in[0,1]$, and let 
$x=(1-\lambda)y+\lambda z$. The level set $K_f^=$ is convex, as $f(x)=\alpha$ holds 
if $f(y)=f(z)=\alpha$. The set $K_f^=$ is an extreme subset of $K_f^\leq$ as
\[
f(x)=(1-\lambda)f(y)+\lambda f(z)<\alpha
\]
holds for all $\lambda\in(0,1)$ if $f(y)<\alpha$ or $f(z)<\alpha$.
\end{proof}
Let $\N_0=\{0,1,2,\ldots\}$.
\par
\begin{Thm}\label{thm:ExtConst}
Let $f:K\to\R\cup\{+\infty\}$ be a generalized affine map. Let $C$ denote a 
sublevel set $C=K_f^\leq$ or a level set $C=K_f^=$ and let $x$ be a point in 
$C$. If the face $F_C(x)$ of $C$ generated by $x$ has dimension $m\in\N_0$, 
then the face $F_K(x)$ of $K$ generated by $x$ has dimension $m$ or $m+1$.
\end{Thm}
\begin{proof}
The inclusion $F_K(x)\supseteq F_C(x)$ holds as $F_K(x)\cap C$ is a face of $C$,
and provides the lower bound of $\dim F_K(x)\geq m$.
\par
Assume that $F_K(x)$ has dimension $m+2$ or larger. We may choose an affine 
subspace $A\subseteq\aff F_K(x)$ of dimension $m+2$ incident with $x$. 
As $K_f$ is an extreme subset of $K$ by Lemma~\ref{lem:GenFinite}
and as $x\in K_f$, Corollary~\ref{cor:extreme} proves $F_K(x)\subseteq K_f$. 
This means that $f$ has finite values on the convex set
\[
X=A\cap F_K(x).
\]
As $x$ lies in the relative algebraic interior of $F_K(x)$ by 
Theorem~\ref{thm:aff} and as $x\in A\subseteq\aff F_K(x)$,
Proposition~\ref{pro:intersection} shows $x\in\ri(X)$ and $\aff(X)=A$.
\par
We extend the affine map $f|_X:X\to\R$ to an affine map $g:A\to\R$, and consider
the linear subspace
\[
L=\{v\in\lin(A): g(y+v)=g(y)\;\forall y\in A\}
\]
of the translation vector space $\lin(A)=\{y-x:y\in A\}$. The space $L$ has 
codimension at most one, so $\dim(L)\geq m+1$. Since $x\in\ri(X)$, 
Lemma~\ref{lem:relint} shows that $x$ generates $X=F_X(x)$ as a face of $X$. 
Then Proposition~\ref{pro:intersection} shows 
\[
x\in\ri(X\cap(x+L))
\quad\text{and}\quad 
\aff(X\cap(x+L))=x+L. 
\]
As $X\cap(x+L)\subseteq C$ and as $x\in\ri(X\cap(x+L))$, Lemma~\ref{lem:relint}, 
part 1), provides the inclusion $X\cap(x+L)\subseteq F_C(x)$. Taking affine hulls, 
we get $x+L\subseteq\aff(F_C(x))$. This proves $\dim(F_C(x))>m$, which completes 
the proof.
\end{proof}
We study intersections of sublevel and level sets. Let $u\subset\N$ be a finite subset, 
let $f_k:K\to\R\cup\{+\infty\}$ be a generalized affine map, and let $\alpha_k\in\R$ 
for all $k\in u$. For any subset $s\subseteq u$ we study the intersection 
\begin{equation}\label{eq:const-multiple}
K_u^s
=\{x\in K: f_k(x)\leq\alpha_k\;\forall k\in u\setminus s
\;\text{and}\; f_k(x)=\alpha_k\;\forall k\in s \}
\end{equation}
of $|u|-|s|$ sublevel sets and $|s|$ level sets. If $(\alpha_k)_{k\in u}$ is 
unspecified and $s,t\subseteq u$, we assume the intersections of sublevel
and level sets $K_u^s$ and $K_u^t$ are taken at the same values of 
$(\alpha_k)_{k\in u}$.
\begin{Lem}\label{lem:sub-level-mixed}
If $t\subseteq s\subseteq u$, then $K_u^s$ is a face of 
$K_u^t$.
\end{Lem}
\begin{proof}
Let $t\subseteq u$, $k\in u\setminus t$, and $s=t\cup\{k\}$. Lemma~\ref{lem:sub-level} 
shows that the intersection $(C_t^t)_{f_k}^=$ of level sets is a face of 
$(C_t^t)_{f_k}^\leq$ for all convex subsets $C\subseteq K$, that is to say, $C_s^s$ is 
a face of $C_s^t$. If $C$ is the intersection $C=K_{u\setminus s}^\emptyset$ of 
sublevel sets, it follows that $(K_{u\setminus s}^\emptyset)_s^s$ is a face of 
$(K_{u\setminus s}^\emptyset)_s^t$. In other words, $K_u^s$ is a face of $K_u^t$. The 
general case follows by induction as faces of faces of a convex set are faces of the 
convex set.
\end{proof}
It is useful to iterate Theorem~\ref{thm:ExtConst}.
\par
\begin{Cor}\label{cor:ExtConst-mult}
Let $\ell\in\N$, let $u=\{1,2,\ldots,\ell\}$, let $s\subseteq u$, and let $x$ be a 
point in the intersection $K_u^s$ of sublevel and level sets from equation 
\eqref{eq:const-multiple}. If the face $F_{K_u^s}(x)$ of $K_u^s$ generated by $x$ 
has dimension $m\in\N_0$, then the dimension of the face $F_K(x)$ of $K$ generated 
by $x$ belongs to the set $\{m,m+1,\ldots,m+\ell\}$.
\end{Cor}
\begin{proof}
This follows from Theorem~\ref{thm:ExtConst} by induction.
\end{proof}
Corollary~\ref{cor:ExtConst-mult} allows us to exploit gaps in the list of the 
dimensions of the faces of $K$.
\par
\begin{Cor}\label{cor:gaps}
Let $m,n\in\N_0$ such that $n>m$, let $M\doteq\{m,m+1,\ldots,n\}$, and let 
$D\subseteq M$. Let $\ell\in\N$ such that $\ell\leq n-m$, let 
$u=\{1,2,\ldots,\ell\}$, let $s\subseteq u$, and let $K_u^s$ be the intersection of 
(sub-) level sets from equation \eqref{eq:const-multiple}.
\begin{enumerate}
\item
If $K$ has no face with dimension in $M\setminus D$ and if the face $F_{K_u^s}(x)$
of $K_u^s$ generated by a point $x\in K_u^s$ has dimension in 
$\{m,m+1,\ldots,n-\ell\}$, then $\,\dim F_K(x)\in D$. 
\item
If $K$ has no face with dimension in $M$, then $K_u^s$ has no face with 
dimension in $\{m,m+1,\ldots,n-\ell\}$.
\end{enumerate}
\end{Cor}
\begin{proof}
Part 1). Let $x$ be a point in $K_u^s$ and let the dimension of the face $F_{K_u^s}(x)$ 
belong to the set $\{m,m+1,\ldots,n-\ell\}$. Corollary~\ref{cor:ExtConst-mult} shows 
that the dimension of the face $F_K(x)$ belongs to $M$, which implies 
$\,\dim\,F_K(x)\in D$ by the assumptions.
\par
Part 2). Let $F$ be a face of $K_u^s$ and let $\,\dim(F)\in\{m,m+1,\ldots,n-\ell\}$. 
As $F$ has finite dimension $m\geq 0$, the relative algebraic interior $\ri(F)$ 
contains a point $x$, see Theorem 6.2 and Theorem 11.6 of \cite{Rockafellar1970}.
Part 3) of Lemma~\ref{lem:relint} proves $F=F_{K_u^s}(x)$ and part 1) of the present 
corollary, with $D=\emptyset$, shows $\,\dim F_K(x)\in\emptyset$. This is a 
contradiction.
\end{proof}
Corollary~\ref{cor:gaps}, part 1), is simplified as follows if $m=0$,
$n=\ell$, and $D=\{0\}$.
\par
\begin{Cor}\label{cor:easy-gaps}
Let $\ell\in\N$, let $u=\{1,\ldots,\ell\}$, and let $K$ have no face with dimension in 
$u$. Then every extreme point of the intersection $K_u^s$ of sublevel and level sets 
is an extreme point of $K$ for all $s\subseteq u$.
\end{Cor}
Corollary~\ref{cor:easy-gaps} is optimal in the sense that if $K$ (or one of its faces) 
has dimension $\ell$, then there are affine functionals and a point $x\in K$
such that $x$ is an extreme point of $K_u^s$ for all $s\subseteq u$, but not an extreme 
point of $K$. For example, the origin is an extreme point of the set
\[
K_u^s=\{(x_1,\ldots,x_\ell)\in\R^\ell : x_k\leq 0\;\forall k\in u\setminus s
\;\text{and}\; x_k=0\;\forall k\in s \}
\]
for all subsets $s\subseteq u$, but not an extreme point of $K=\R^\ell$. 
\par
We apply the results to pyramids. Let $o\in V$ be a point outside of the affine hull
of $K$. The {\em pyramid} with apex $o$ over a nonempty subset $F\subseteq K$  
is the union of all closed segments joining points in $F$ with $o$,
\[
\P(F,o)=\bigcup_{x\in F}[x,o].
\]
In addition, we define $\P(\emptyset,o)=\{o\}$. We frequently write $\P(F)$ instead of
$\P(F,o)$. Note that the pyramid over a convex subset $F\subseteq K$ is the convex hull
of $F\cup\{o\}$. For every $x\in\P(K,o)\setminus\{o\}$ we denote by $\hat{x}$ the
point of $K$ that is incident with the line through $o$ and $x$.
\par
\begin{Lem}\label{lem:FacesPyr}
The set of faces of $\P(K,o)$ is the union of the set $\cF_1$ of faces of $K$ and
the set of pyramids $\cF_2=\{\P(F,o) : F\in\cF_1\}$.
\end{Lem}
\begin{proof}
Let $G$ be a face of $\P(K)$. First, we show $o\not\in G\Rightarrow G\in\cF_1$.
As $G$ is an extreme set, $x\in G$ and $x\not\in K\cup\{o\}$ imply $o,\hat{x}\in G$.
On the contrapositive, if $o\not\in G$ then $G\subseteq K$. As $G$ is a face of
$\P(K)$ it is {\em a fortiori} a face of $K$. This shows $G\in\cF_1$.
Secondly, we show $o\in G\Rightarrow G\in\cF_2$. It is easy to see that $G$ is
the pyramid over some subset $F\subseteq K$. Indeed, with any point $x\in K$ the
convex set $G$ contains also the segment $[x,o]$. Moreover, if $G$ contains a point
$x\not\in K\cup\{o\}$, then $o,\hat{x}\in G$. This proves $G=\P(F)$ for some subset
$F\subseteq K$. Since $F=G\cap K$, the set $F$ is a face of $K$. This shows
$G\in\cF_2$.
\par
Each element of $\cF_1\cup\cF_2$ is a face. Let $p,p_1,p_2$ be any three points
in the pyramid $\P(K)$ such that $p$ lies in the open segment $(p_1,p_2)$. We may
write $p=(1-\lambda)p_1+\lambda p_2$, where
\[
p=(1-\eta)x+\eta o
\qquad\text{and}\qquad
p_i=(1-\mu_i)x_i+\mu_i o,
\quad i=1,2,
\]
$x,x_1,x_2\in K$, $\eta,\mu_1,\mu_2\in[0,1]$, and $\lambda\in(0,1)$. Then
\begin{equation}\label{eq:3points}
(1-\eta) x+\eta o=(1-\lambda)(1-\mu_1)x_1+\lambda(1-\mu_2)x_2
+((1-\lambda)\mu_1+\lambda\mu_2)o.
\end{equation}
As $x,x_1,x_2\in\aff(K)$ and $o\not\in\aff(K)$, equation \eqref{eq:3points} shows
$\eta=(1-\lambda)\mu_1+\lambda\mu_2$. The convex sets $K$ and $\{o\}$ are extreme
subsets (and hence faces) of $\P(K)$ as they correspond to the extreme values
$\eta=0$ and $\eta=1$, which require $\mu_1=\mu_2=0$ and $\mu_1=\mu_2=1$,
respectively. That $K$ is a face of $\P(K)$ implies that every face of $K$ is a
face of $\P(K)$, too. Let us show that $\P(F)$ is an extreme subset of $\P(K)$ for
all faces $F$ of $K$. Let $p\in\P(F)$. We may assume $\eta<1$ as $o$ is an
extreme point. The equation \eqref{eq:3points} simplifies then to
\begin{equation}\label{eq:other3points}
x=\tfrac{(1-\lambda)(1-\mu_1)}{1-\eta}x_1+\tfrac{\lambda(1-\mu_2)}{1-\eta}x_2.
\end{equation}
If $\mu_1=1$, then $x_2=x$ follows and hence $p_2\in[x,o]\subseteq\P(F)$. Similarly,
$\mu_2=1$ implies $p_1\in\P(F)$. If $\mu_1<1$ and $\mu_2<1$, then
\eqref{eq:other3points} shows that $x\in(x_1,x_2)$. As $x\in F$ and as $F$ is an
extreme subset of $K$, we obtain $x_1,x_2\in F$, hence $p_1,p_2\in\P(F)$. This
proves that $\P(F)$ is a face of $\P(K)$.
\end{proof}
By Lemma~\ref{lem:FacesPyr}, the face of the pyramid $\P(K,o)$ generated by a point is
\begin{equation}\label{eq:face-pyr}
F_{\P(K,o)}(x) = \left\{
\begin{array}{ll}
\{o\} & \text{if $x=o$},\\
F_K(x) &  \text{if $x\in K$},\\
\P(F_K(\hat{x}),o) & \text{else},
\end{array}\right.
\qquad \text{for all $x\in\P(K,o)$.}
\end{equation}
Equation~\eqref{eq:face-pyr} allows us to simplify Corollary~\ref{cor:ExtConst-mult} 
when applied to pyramids.
\par
\begin{Cor}\label{cor:ExtPyrConst}
Let $o\in V$ be a point outside of the affine hull of $K$. Let $\ell\in\N$, let 
$u=\{1,\ldots,\ell\}$, let $s\subseteq u$, and let $x$ be a point in the intersection 
$\P(K,o)_u^s$ of sublevel and level sets. If the face $F_{\P(K,o)_u^s}(x)$ of 
$\P(K,o)_u^s$ generated by $x$ has dimension $m\in\N_0$, then exactly one of the 
following cases applies.
\begin{enumerate}
\item
The point $x$ is the apex $o$, an extreme point of $\P(K,o)_u^s$ 
and $\P(K,o)$.
\item
The point $x$ lies in $K$ and generates the face $F_{\P(K,o)}(x)=F_K(x)$ of the 
pyramid $\P(K,o)$.
The dimension of $F_K(x)$ lies in $\{m,m+1,\ldots,m+\ell\}$.
\item
The point $x$ lies outside of $K\cup\{o\}$ and generates the face $\P(F_K(\hat{x}),o)$
of $\P(K,o)$. The dimension of the face $F_K(\hat{x})$ of $K$ generated by $\hat{x}$ 
lies in $\{m-1,m,\ldots,m+\ell-1\}$ if $m\geq 1$ and in $\{0,1,\ldots,\ell-1\}$ 
if $m=0$.
\end{enumerate}
\end{Cor}
\begin{proof}
The claim follows from Corollary~\ref{cor:ExtConst-mult} and equation
\eqref{eq:face-pyr}. If $m=0$, then the dimension $m-1=-1$ of $F_K(\hat{x})$ is 
excluded from case 3) as $F_K(\hat{x})$ is nonempty.
\end{proof}
We discuss the pyramidal counterpart to Corollary~\ref{cor:easy-gaps}.
\par
\begin{Cor}\label{cor:gap-pyr}
Let $o\in V$ be a point outside of the affine hull of $K$. Let $\ell\in\N$, let 
$u=\{1,\ldots,\ell\}$, let $s\subseteq u$, and let $K$ have no face with dimension 
in $u$. Then every extreme point of $\P(K,o)_u^s$ is a convex combination of one 
extreme point of $K$ and of the apex $o$ of the pyramid $\P(K,o)$.
\end{Cor}
\begin{proof}
The claim follows from Corollary~\ref{cor:ExtPyrConst} when $m=0$. Let $x$ be an 
extreme point of $\P(K,o)_u^s$. Case 1) of Corollary~\ref{cor:ExtPyrConst} is 
consistent with the claim. In case 2) we have $x\in K$ and 
$\dim F_K(x)\in\{0,1,\ldots,\ell\}$. The assumption $\dim(F_K(x))\not\in u$ implies 
that $x$ is an extreme point of $K$. In case 3) we have 
$F_{\P(K,o)}(x)=\P(F_K(\hat{x}),o)$ and $\dim(F_K(\hat{x}))\in\{0,\ldots,\ell-1\}$. 
The assumption $\dim(F_K(\hat{x}))\not\in u$ shows that $\hat{x}$ is an extreme 
point of $K$. Hence, $x$ is the convex combination $x=(1-\lambda)\hat{x}+\lambda o$ 
for some $\lambda\in(0,1)$.
\end{proof}
%
%%%%%%%%%%%%%%%%%%%%%%%%%%%%%%%%%%%%%%%%%%%%%%%%%%%%%%%%%%%%%%%%%%%%%%%%%%%%
%%%%%%%%%%%%%%%%%%%%%%%%%%%%%%%%%%%%%%%%%%%%%%%%%%%%%%%%%%%%%%%%%%%%%%%%%%%%
%%%%%%%%%%%%%%%%%%%%%%%%%%%%%%%%%%%%%%%%%%%%%%%%%%%%%%%%%%%%%%%%%%%%%%%%%%%%
%%%%%%%%%%%%%%%%%%%%%%%%%%%%%%%%%%%%%%%%%%%%%%%%%%%%%%%%%%%%%%%%%%%%%%%%%%%%
%%%%%%%%%%%%%%%%%%%%%%%%%%%%%%%%%%%%%%%%%%%%%%%%%%%%%%%%%%%%%%%%%%%%%%%%%%%%
%
\section{Extreme Points of Quantum States under Expected Value Constraints}
\label{sec:extreme-points-quantum}
In the remainder of the article we explore expected value functionals on the set 
of quantum states. These functionals are generalized affine maps. In the present 
section we apply the above findings to pairs of expected value functionals. 
We also discuss the failure of analogous assertions for triples of expected value 
functionals and for the set of classical states.
\par
Let $\cH$ be a separable Hilbert space with inner product 
$\braket{\,\cdot\,|\,\cdot\,}$. The space $\fT$ of trace-class operators on 
$\cH$ is a Banach space with the trace norm $\|\cdot\|_1$. The real Banach space of 
self-adjoint trace-class operators contains the closed convex cone $\fT^+$ of 
positive trace-class operators, which contains the closed convex sets
$\fT^1=\fT^1(\cH)$ of positive trace-class operators with trace at most one and
$\fS=\fS(\cH)$ of positive trace-class operators with trace equal one called 
quantum states or density operators. Note that $\fT^1=\P(\fS,0)$ is the pyramid over 
$\fS$ with apex zero.
\par
We define a constraint on $\fT^+$ using a (possibly unbounded) positive operator $H$ 
on $\cH$. We approximate $H$ by the sequence $HP_n$ of bounded operators, where 
$P_n=\int_0^n dE_H(\lambda)$ is the spectral projector of $H$ corresponding to $[0,n]$ 
and $E_H$ is a spectral measure on the Borel $\sigma$-algebra of $[0,\infty)$, see for 
example \cite{Schmuedgen2012}. We define the functional
\[
f_H: \fT^+\to[0,+\infty],
\qquad
A\mapsto\Tr HA=\lim_{n\to\infty}\Tr(HP_nA).
\]
The number $\Tr H\rho$ is the {\em expected value} of the observable associated to $H$ 
if $\rho\in\fS$ is the state of the quantum system. The map $f_H$ is lower 
semicontinuous as $f_H(A)=\sup_{n\in\N}\Tr(HP_nA)$ for all $A\in\fT^+$. Since $H$ is a 
positive operator, the map $f_H$ is a generalized affine map in the sense of 
Section~\ref{sec:extreme-points}. This remains true if we replace $H$ with a 
self-adjoint, lower-bounded operator on $\cH$. Similarly, all assertions below remain 
valid if we replace positive operators with self-adjoint, lower-bounded operators.
\par
We study constraints imposed by several operators using a notation similar to 
equation \eqref{eq:const-multiple}. Let $\ell\in\N$, let $H_k$ be a positive 
operator on $\cH$, and let $E_k\in\R$ for all $k\in u=\{1,\ldots,\ell\}$. For each 
subset $s\subseteq u$, we define the intersections 
\begin{align}\label{eq:sublevelS}
& \fS_{H_1,E_1,H_2,E_2,\ldots,H_\ell,E_\ell}^s\\\nonumber
& \; = \;
\{\rho\in\fS(\cH) : \Tr H_k\rho\leq E_k\;\forall k\in u\setminus s
\;\text{and}\; \Tr H_k\rho=E_k\;\forall k\in s \}
\end{align}
and 
\begin{align}\label{eq:sublevelT1}
& (\fT^1)_{H_1,E_1,H_2,E_2,\ldots,H_\ell,E_\ell}^s\\\nonumber
& \; = \;
\{\rho\in\fT^1(\cH) : \Tr H_k\rho\leq E_k\;\forall k\in u\setminus s
\;\text{and}\; \Tr H_k\rho=E_k\;\forall k\in s \}
\end{align}
of $\ell-|s|$ sublevel sets and $|s|$ level sets. We simplify the notation for 
{\em sublevel sets} by writing
\[
\fS_{H_1,E_1,H_2,E_2,\ldots,H_\ell,E_\ell}
=\fS_{H_1,E_1,H_2,E_2,\ldots,H_\ell,E_\ell}^\emptyset
\]
and
\[
\fT^1_{H_1,E_1,H_2,E_2,\ldots,H_\ell,E_\ell}
=(\fT^1)_{H_1,E_1,H_2,E_2,\ldots,H_\ell,E_\ell}^\emptyset.
\]
The intersections $\fS_{H_1,E_1,H_2,E_2,\ldots,H_\ell,E_\ell}$ and 
$\fT^1_{H_1,E_1,H_2,E_2,\ldots,H_\ell,E_\ell}$ of sublevel sets are closed sets as 
the map $\fT^+\to[0,+\infty]$, $A\mapsto\Tr H_kA$ is lower semi-continuous for all
$k=1,\ldots,\ell$.
\par
It is well known that the set of extreme points $\,\ext(\fS)$ of the set of quantum 
states $\fS(\cH)$ consists of the projectors of rank one, called {\em pure states}. 
The finite-dimensional faces are isometric to the sets $\fS(\C^d)$ for all 
$d\leq\dim(\cH)$.
\par
\begin{Lem}\label{lem:dimensions-q}
If a face of the set of quantum states $\fS(\cH)$ has finite dimension $n<\infty$, 
then $n=d^2-1$ for some $d\in\N$.
\end{Lem}
\begin{proof}
Theorem~4.6 in \cite{AlfsenShultz2001} proves that the finite-dimensional closed 
faces of $\fS(\cH)$ have dimensions $d^2-1$, $d\in\N$. The claim then follows from 
showing that every nonempty, finite-dimensional face $F$ of $\fS(\cH)$ is closed. 
As $\dim(F)<\infty$, the closure $\overline{F}$ is included in $\aff(F)$ and 
the relative algebraic interior $\ri(F)$ contains a point $x$. Let $y\in\overline{F}$ 
be arbitrary. As $y\in\aff(F)$, the definition of the relative algebraic interior
shows that there is a point $z\in F$ such that $x$ lies in the open segment $(y,z)$.
Since $\fS(\cH)$ is closed, we have $y\in\fS(\cH)$. As $F$ is an extreme
subset of $\fS(\cH)$, this shows $y\in F$ and completes the proof.
\end{proof}
Taking into account the list of dimensions from Lemma~\ref{lem:dimensions-q},
and invoking Corollary~\ref{cor:easy-gaps} and Corollary~\ref{cor:gap-pyr}, we 
obtain the following assertion.
\par
\begin{Thm}\label{thm:ExtSE+}
Let $H_1,H_2$ be arbitrary positive operators on $\cH$, let $E_1,E_2\in\R$, and let
$s\subseteq\{1,2\}$. Then all extreme points of the set 
$\fS_{H_1,E_1,H_2,E_2}^s$ are pure states. All extreme points of the set 
$(\fT^1)_{H_1,E_1,H_2,E_2}^s$ have rank at most one.
\end{Thm}
Theorem~\ref{thm:ExtSE+} implies Corollary~\ref{cor:ExtSE} below by taking 
$H_2=H_1$ and $E_2=E_1$. In the sequel, we will omit further mention of similar 
reductions from two to one operators. 
\par
\begin{Cor}\label{cor:ExtSE}
Let $H$ be an arbitrary positive operator on $\cH$, let $E\in\R$, and let 
$s\subseteq\{1\}$. Then all extreme points of the set $\fS_{H,E}^s$ are pure states. 
All extreme points of the set $(\fT^1)_{H,E}^s$ have rank at most one.
\end{Cor}
Let $\cH=\C^d$ for some $d\in\N$. The set of quantum states $\fS(\C^d)$ is a 
compact, convex set, which is a base of the cone $\fT^+(\C^d)$ of positive 
semidefinite matrices. If $H\in\fT^+(\C^d)$ then
\[
f_H: \fT^+(\C^d)\to[0,+\infty),
\qquad
A\mapsto\Tr HA
\]
is a continuous, affine map. 
\par
\begin{Cor}\label{cor:finite-dim}
Let $\cH=\C^d$ for some $d\in\N$. Let $H_1,H_2\in\fT^+(\C^d)$ be arbitrary positive 
semidefinite matrices, let $E_1,E_2\in\R$, and let $s\subseteq\{1,2\}$. Then the 
intersection $\fS_{H_1,E_1,H_2,E_2}^s$ of sublevel and level sets is a compact, 
convex set. Every state $\rho\in\fS_{H_1,E_1,H_2,E_2}^s$ can be represented as
\begin{equation}\label{eq:Memarzadeh-Mancini}\textstyle
\rho=\sum_{i=1}^{d^2} p_i\sigma_i,
\end{equation}
where $\{p_i\}_{i=1}^{d^2}$ is a probability distribution and 
$\{\sigma_i\}_{i=1}^{d^2}\subseteq\fS_{H_1,E_1,H_2,E_2}^s$ is a set of pure states. 
\end{Cor}
\begin{proof}
The convex set $\fS_{H_1,E_1,H_2,E_2}^s$ is compact as the set of quantum states 
$\fS(\C^d)$ is compact and as 
\[
\fT^+(\C^d)\to[0,+\infty),
\qquad
A\mapsto\Tr H_i A,
\qquad i=1,2
\]
are continuous maps. Carath\'eodory's theorem asserts that every point in a compact, 
convex subset $C$ of $\R^n$ is a convex combination of at most $n+1$ extreme 
points of $C$, see for example \cite{Roy1987,Schneider2014}. The claim follows 
as the extreme points of $\fS_{H_1,E_1,H_2,E_2}^s$ are pure states by 
Theorem~\ref{thm:ExtSE+}, and as $\dim\fS(\C^d)=d^2-1$.
\end{proof}
The assertion \eqref{eq:Memarzadeh-Mancini} of Corollary~\ref{cor:finite-dim}
for the level set $\fS_{H,E}^{\{1\}}$ is proved in \cite{Man}.
\par
\begin{Rem}\label{rem:three-obs}
If more than two positive operators are employed, the assertion analogous to
Theorem \ref{thm:ExtSE+} is not valid. Perhaps, the simplest example is the Hilbert 
space $\cH=\C^2$ and positive semidefinite matrices $H_1=\id+X$, $H_2=\id+Y$, and 
$H_3=\id+Z$, where $\id=\left(\begin{smallmatrix}1&0\\0&1\end{smallmatrix}\right)$
is the identity matrix and 
$X=\left(\begin{smallmatrix}0&1\\1&0\end{smallmatrix}\right)$,
$Y=\left(\begin{smallmatrix}0&-\ii\\\ii&0\end{smallmatrix}\right)$, and
$Z=\left(\begin{smallmatrix}1&0\\0&-1\end{smallmatrix}\right)$ 
are the Pauli matrices. If $E_1=E_2=E_3=1$, then for all subsets 
$s\subseteq u=\{1,2,3\}$ the set
\begin{align*}
\fS^s & = \; \fS_{H_1,1,H_2,1,H_3,1}^s\\
& = \;
\{\rho\in\fS(\C^2) : \Tr H_k\rho\leq E_k\;\forall k\in u\setminus s
\;\text{and}\;\Tr H_k\rho=E_k\;\forall k\in s \}
\end{align*}
is a spherical sector of the Bloch ball $\fS(\C^2)$. Theorem \ref{thm:ExtSE+} fails 
as the trace state $\tfrac{1}{2}\id$ is an extreme point of $\fS^s$ of rank two.
\par
Theorem~\ref{thm:sup-pure-2obs+} below allows us to express the suprema of certain 
functions as suprema over pure states. However, this is not possible for more than 
two positive operators. Consider the map
\[
f:\fS^s\to\R,
\quad
\rho\mapsto\Tr(X+Y+Z)\rho.
\]
The domain $\fS^s$ is the intersection of $t=3-|s|$ sublevel sets and $|s|$ level 
sets. The image $f(\fS^s)$ is the interval $[-\sqrt{t},0]$. The image of the set 
of pure states in $\fS^s$ under $f$ is the interval $[-\sqrt{t},-1]$ if $t\geq 1$ 
and is empty if $t=0$. If $t\geq 1$, then the maximum on the set of pure states is 
attained at $t$ of the pure states 
\[
\tfrac{1}{2}(\id-X), \quad 
\tfrac{1}{2}(\id-Y), \quad\text{and}\quad
\tfrac{1}{2}(\id-Z).
\]
In any case, the maximum $f(\tfrac{1}{2}\id)=0>-1$ of $f$ is neither equal to nor 
approximated by the values of $f$ at pure states in $\fS^s$. Similarly, it is easy 
to check that for all $s\subseteq u$ 
\[
\conv\{\rho\in\fS^s : \text{$\rho$ is a pure state} \}
=  \{\rho\in\fS^s : f(\rho)\leq-1\}.
\]
No state $\rho\in\fS^s$ with $f(\rho)>-1$ can be the barycenter of a probability 
measure supported on the set of pure states in $\fS^s$. This shows that 
Theorem~\ref{thm:KMC} and Corollary~\ref{cor:C++} below fail for more than 
two positive operators.
\end{Rem}
The classical analogues of our results fail as the set of classical states 
has one-dimension faces. 
\par
\begin{Rem}[Classical states]
The set of classical states on the Hilbert space $\C^3$ with respect to an 
orthonormal basis $e_1,e_2,e_3$ of $\C^3$ is
\[
\fP_3
=\left\{p(1)\sigma_1
+p(2)\sigma_2
+p(3)\sigma_3 : p\in\Delta_{\{1,2,3\}}\right\}.
\]
Here, $\Delta_{\{1,2,3\}}$ is the simplex of probability densities introduced in
Example~\ref{exa:discrete-pms} and $\sigma_n$ is the projector onto the line spanned 
by $e_n$ for all $n=1,2,3$. The set $\fP_3$ is a triangle with extreme points 
$\sigma_1,\sigma_2,\sigma_3$. The sublevel and level set of
\[
f:\fP_3\to\R,
\quad
\rho\mapsto\Tr(\sigma_3\rho)=\braket{e_3|\rho e_3}
\]
at $\alpha=\tfrac{1}{2}$ is denoted in equation \eqref{eq:simple-const},
respectively, as
\[
(\fP_3)_f^\leq =\{\rho\in\fP_3:f(\rho)\leq\tfrac{1}{2}\}
\quad\text{and}\quad
(\fP_3)_f^= =\{\rho\in\fP_3:f(\rho)=\tfrac{1}{2}\}.
\]
\par
The level set $(\fP_3)_f^=$ is the segment $[\rho_1,\rho_2]$ and has the  
extreme points $\rho_i=\tfrac{1}{2}(\sigma_i+\sigma_3)$, $i=1,2$. By 
Lemma~\ref{lem:sub-level}, the points $\rho_1,\rho_2$ are also extreme points of 
the sublevel set $(\fP_3)_f^\leq$. The analogue of Corollary~\ref{cor:ExtSE} fails 
for classical states as the points $\rho_1,\rho_2$ have rank two despite the fact 
that they are extreme points of $(\fP_3)_f^=$ and $(\fP_3)_f^\leq$.
\par
The sublevel set $(\fP_3)_f^\leq$ contains only two pure states, namely $\sigma_1$ 
and $\sigma_2$. Hence, only the states on the segment $[\sigma_1,\sigma_2]$ can be
represented as convex combinations of pure states from $(\fP_3)_f^\leq$. In 
particular, the analogue of Corollary~\ref{cor:C+} fails: It is impossible to 
represent any state from $(\fP_3)_f^=$ as the barycenter of pure states from
$(\fP_3)_f^=$. Similarly, the analogue of Theorem~\ref{thm:sup-pure-2obs+} 
fails: The supremum $f(\rho_1)=1/2$ of $f$ on 
$(\fP_3)_f^\leq$ is not attained (neither approximated) by pure states as 
$f(\sigma_1)=f(\sigma_2)=0$ holds for the sole pure states in $(\fP_3)_f^\leq$.
\end{Rem}
%
%%%%%%%%%%%%%%%%%%%%%%%%%%%%%%%%%%%%%%%%%%%%%%%%%%%%%%%%%%%%%%%%%%%%%%%%%%%%
%%%%%%%%%%%%%%%%%%%%%%%%%%%%%%%%%%%%%%%%%%%%%%%%%%%%%%%%%%%%%%%%%%%%%%%%%%%%
%%%%%%%%%%%%%%%%%%%%%%%%%%%%%%%%%%%%%%%%%%%%%%%%%%%%%%%%%%%%%%%%%%%%%%%%%%%%
%%%%%%%%%%%%%%%%%%%%%%%%%%%%%%%%%%%%%%%%%%%%%%%%%%%%%%%%%%%%%%%%%%%%%%%%%%%%
%%%%%%%%%%%%%%%%%%%%%%%%%%%%%%%%%%%%%%%%%%%%%%%%%%%%%%%%%%%%%%%%%%%%%%%%%%%%
%
\section{Pure-State Decomposition Theorem}
\label{sec:decomposition}
Let $H_1$ and $H_2$ be positive operators on a separable Hilbert space $\cH$. If 
$\dim(\cH)<\infty$, then Corollary~\ref{cor:finite-dim} above provides a pure-state 
decomposition for the intersection $\fS_{H_1,E_1,H_2,E_2}^s$ of sublevel and level 
sets for all $s\subseteq\{1,2\}$, see \eqref{eq:sublevelS} for the notation. If 
$\dim(\cH)=\infty$, we need to differentiate between sublevel and level sets.
Despite the fact that the former are closed (as the expected value functionals are 
lower semicontinuous) and $\mu$-compact while the latter are not even closed,
we prove pure-state decompositions for both.
\par
We begin with sublevel sets. If $H_1$ (or $H_2$) is a positive operator with a 
discrete spectrum of finite multiplicity, then the intersections 
$\fS_{H_1,E_1,H_2,E_2}$ and $\fT^1_{H_1,E_1,H_2,E_2}$ of sublevel sets are compact. 
Indeed, it has been shown in \cite{Holevo2004} that $\fS_{H_1,E_1}$ is 
compact\footnote{Recall that if a positive operator $H$ on an infinite dimensional
Hilbert space $\cH$ has a discrete spectrum of finite multiplicity, then there is 
sequence of non-negative numbers $(\lambda_n)_{n\in\N}$ and an orthonormal basis 
$\{e_n:n\in\N\}$ of $\cH$ such that $\lim_{n\to\infty}\lambda_n=\infty$ and 
$He_n=\lambda e_n$ for all $n\in\N$. See for example \cite{Schmuedgen2012},
Corollary~5.11 and Proposition~5.12.}. 
It follows that $\fS_{H_1,E_1,H_2,E_2}=\fS_{H_1,E_1}\cap\fS_{H_2,E_2}$ is compact as 
$\fS_{H_2,E_2}$ is closed. Similarly, one can show that $\fT^1_{H_1,E_1,H_2,E_2}$ is 
compact by using Proposition~11 in \cite[Appendix]{AQC}.
\par
If $H_1$ and $H_2$ are arbitrary positive operators, the sets $\fS_{H_1,E_1,H_2,E_2}$ 
and $\fT^1_{H_1,E_1,H_2,E_2}$ are closed but not {\em compact}. Yet, they are 
{\em $\mu$-compact} by Proposition 2 in \cite{HolevoShirokov2006} and Proposition 4 
in \cite{P&Sh}, respectively. Proposition~5 in \cite{P&Sh} provides generalized 
assertions of Krein-Milman's theorem and of Choquet's theorem for $\mu$-compact sets. 
We employ Theorem~\ref{thm:ExtSE+} to make these assertions more explicit.
\par
\begin{Thm}\label{thm:KMC}
Let $H_1,H_2$ be arbitrary positive operators on $\cH$ and $E_1,E_2$ nonnegative 
numbers such that the intersection $\fS_{H_1,E_1,H_2,E_2}$ of sublevel sets is 
nonempty. Then the set $\,\ext\fS_{H_1,E_1,H_2,E_2}$ of extreme points is equal to 
the set of pure states $\fS_{H_1,E_1,H_2,E_2}\cap\ext\fS(\cH)$, which is nonempty
and closed.
\begin{enumerate}
\item[A (Krein-Milman's theorem).]
The set $\,\fS_{H_1,E_1,H_2,E_2}$ is the closure of the convex hull of $\,\ext\fS_{H_1,E_1,H_2,E_2}$.
\item[B (Choquet's theorem).]
Any state $\rho\in\fS_{H_1,E_1,H_2,E_2}$ can be represented as the barycenter
$\rho=\int \sigma\mu(d\sigma)$ of some Borel probability measure $\mu$ supported
by $\,\ext\fS_{H_1,E_1,H_2,E_2}$.
\end{enumerate}
\end{Thm}

\begin{proof}
Theorem~\ref{thm:ExtSE+} shows that the set of extreme points 
$\,\ext\fS_{H_1,E_1,H_2,E_2}$ is the intersection of $\fS_{H_1,E_1,H_2,E_2}$ and 
the set of pure states $\,\ext\fS(\cH)$. As both sets are closed, their intersection 
is closed.
\par
The remaining assertions follow from Proposition 5 in \cite{P&Sh} as
$\fS_{H_1,E_1,H_2,E_2}$ is $\mu$-compact and since $\,\ext\fS_{H_1,E_1,H_2,E_2}$
is closed.
\end{proof}
Note that the closedness of the set $\,\ext\fS_{H_1,E_1,H_2,E_2}$ is not obvious
even in the case when both operators $H_1$ and $H_2$ have discrete spectrum
or in the case of $\dim(\cH)<\infty$. The closedness of the set of extreme points
is necessary for the stability \cite{Papadopoulou1977,P&Sh} of
$\,\fS_{H_1,E_1,H_2,E_2}$.
\par
\begin{Que}
Under which conditions on the operators $H_1$ and $H_2$ can part B of 
Theorem~\ref{thm:KMC} be strengthened to the statement that any state in 
$\fS_{H_1,E_1,H_2,E_2}$ is a \emph{countable} convex combination of pure states in 
$\fS_{H_1,E_1,H_2,E_2}$? This and the arguments of Corollary~\ref{cor:C+} below 
would imply that any state with finite expected values of $H_1$ and $H_2$ is a 
\emph{countable} convex combination of pure states with the same expected values.
\end{Que}
Pure-state decompositions are more subtle for level than sublevel sets. 
\par
\begin{Cor}\label{cor:C++}
Let $H_1,H_2$ be arbitrary positive operators on $\cH$, let $E_1,E_2$ be real
numbers, let $s\subseteq\{1,2\}$, and let $\rho$ lie in the intersection 
$\fS_{H_1,E_1,H_2,E_2}^s$ of sublevel and level sets. Then $\rho$ can be 
represented as the barycenter
\begin{equation}\label{u-rep+}
\rho=\int \sigma\mu(d\sigma)
\end{equation}
of some Borel probability measure $\mu$ supported by $\,\ext\fS_{H_1,E_1,H_2,E_2}$
such that $\mu(\ext\fS_{H_1,E_1,H_2,E_2}^s)=1$.
\end{Cor}
\begin{proof} 
The assertion B of Theorem~\ref{thm:KMC}
implies that equation \eqref{u-rep+} holds for some probability measure $\mu$
supported by the set $\,\ext\fS_{H_1,E_1,H_2,E_2}$. Since the function
$\fS_{H_1,E_1,H_2,E_2}\to[0,+\infty]$, $\sigma\mapsto\Tr H_k\sigma$ is affine 
and lower semicontinuous, and since the intersection $\fS_{H_1,E_1,H_2,E_2}$ of 
sublevel sets is closed, bounded, and convex, we have 
(see, f.i., \cite[the Appendix]{EM})
\begin{equation}\label{eq:cor12help}
\int \Tr(H_k\sigma)\,\mu(d\sigma)=\Tr H_k\rho,
\qquad k=1,2.
\end{equation}
If $\Tr H_k\rho=E_k$ holds for $k\in\{1,2\}$, then equation \eqref{eq:cor12help}
implies $\Tr H_k\sigma=E_k$ for $\mu$-almost all $\sigma$ as $\Tr H_k\sigma\leq E_k$ 
holds for all $\sigma$ in the support of $\mu$.
\end{proof}

\begin{Cor}\label{cor:C+}
Let $H_1$ and $H_2$ be arbitrary positive operators on $\cH$. Any state $\rho$ such 
that $\Tr H_k\rho=E_k<+\infty$, $k=1,2$, can be represented as
\begin{equation}\label{u-rep}
\rho=\int \sigma\mu(d\sigma),
\end{equation}
where $\mu$ is a Borel probability measure supported by pure states such that
$\Tr H_k\sigma=E_k$, $k=1,2$, for $\mu$-almost all $\sigma$.\smallskip
\end{Cor}
\begin{proof} 
Corollary~\ref{cor:C+} is the case $s=\{1,2\}$ of Corollary~\ref{cor:C++}.
\end{proof}
Theorem~\ref{thm:KMC} and its Corollaries~\ref{cor:C++} and~\ref{cor:C+} are not 
valid for more than two operators, as the intersection
$\fS_{H_1,E_1,H_2,E_2,\ldots,H_\ell,E_\ell}^s$ of sublevel and level sets may have 
extreme points that are no pure states if $\ell\geq 3$. See Remark~\ref{rem:three-obs} 
for an example.
\par
A Borel probability measure supported on pure states is known as a
{\em generalized ensemble} of pure states \cite{HolevoShirokov2006}, and its
barycenter as a {\em continuous convex combination} of pure states.
The probability measure $\mu$ in part B of Theorem~\ref{thm:KMC} is a
generalized ensemble of pure states with bounded expected values.
In the strict sense, the probability measure $\mu$ in Corollary~\ref{cor:C+}
is not a generalized ensemble of pure states with fixed expected values,
as the support of $\mu$ may contain a set of $\mu$-measure zero where
one of the expected values could be smaller than the fixed value.
\par
\begin{Exa}[On pure-state decomposition of bipartite states]
If quantum systems $A$ and $B$ are described by Hilbert spaces $\cH_A$ and $\cH_B$, 
then the bipartite system $AB$ is described by the tensor product of these spaces, 
i.e. $\cH_{AB}\doteq\cH_A\otimes\cH_B$. A state in $\fS(\cH_{AB})$ is denoted by 
$\rho_{AB}$, its marginal states\footnote{Here $\Tr_{\cH_X}$ 
denotes the partial trace over the space $\cH_X$.}
$\Tr_{\cH_B}\rho_{AB}$ and $\Tr_{\cH_A}\rho_{AB}$ are denoted, respectively, by 
$\rho_{A}$ and $\rho_{B}$. See for example \cite{H-SCI,Wilde}.
\end{Exa}
Corollary \ref{cor:C+} implies the following.
\par
\begin{Cor}\label{new-c}
Let $H_A$ and $H_B$ be arbitrary positive operators on $\cH_A$ and $\cH_B$ correspondingly. Any state $\rho_{AB}$ such that
$\Tr H_A\rho_A=E_A<+\infty$ and $\Tr H_B\rho_B=E_B<+\infty$  can be represented as
\begin{equation*}%\label{AB-rep}
\rho_{AB}=\int \sigma_{\!AB}\,\mu(d\sigma_{\!AB}),
\end{equation*}
where $\mu$ is a Borel probability measure supported by pure states in $\fS(\cH_{AB})$ such that
$\Tr H_A\sigma_A=E_A$ and $\Tr H_B\sigma_B=E_B$ for $\mu$-almost all $\sigma_{AB}$.

If $\dim\cH_A=d_A<+\infty$ and $\dim\cH_B=d_B<+\infty$ then the state $\rho_{AB}$ can be represented as
\[\textstyle
\rho_{AB}=\sum_{k=1}^{d^2_Ad^2_B} p_k\sigma^k_{AB},
\]
where $\{p_k\}$ is a probability distribution and 
$\{\sigma^k_{AB}\}$ is a set of pure states such that $\Tr H_A\sigma^k_A=E_A$ and $\Tr H_B\sigma^k_B=E_B$ for all $k$.
\end{Cor}
\begin{Que}
If $H_A$ and $H_B$ are Hamiltonians of systems $A$ and $B$, then Corollary \ref{new-c} 
states that any bipartite state with finite marginal energies can be decomposed into 
pure states with the same marginal energies. An interesting open question is the possibility of a similar decomposition of a state of a composite quantum system 
consisting of more than two subsystems.  
\end{Que}
%
%%%%%%%%%%%%%%%%%%%%%%%%%%%%%%%%%%%%%%%%%%%%%%%%%%%%%%%%%%%%%%%%%%%%%%%%%%%%
%%%%%%%%%%%%%%%%%%%%%%%%%%%%%%%%%%%%%%%%%%%%%%%%%%%%%%%%%%%%%%%%%%%%%%%%%%%%
%%%%%%%%%%%%%%%%%%%%%%%%%%%%%%%%%%%%%%%%%%%%%%%%%%%%%%%%%%%%%%%%%%%%%%%%%%%%
%%%%%%%%%%%%%%%%%%%%%%%%%%%%%%%%%%%%%%%%%%%%%%%%%%%%%%%%%%%%%%%%%%%%%%%%%%%%
%%%%%%%%%%%%%%%%%%%%%%%%%%%%%%%%%%%%%%%%%%%%%%%%%%%%%%%%%%%%%%%%%%%%%%%%%%%%
%
\section{Applications to Quantum Information Theory}
\label{sec:applications}
In this section we consider some applications of our main results in quantum
information theory and mathematical physics. These applications are based on
the following observation.
\par
\begin{Thm}\label{thm:sup-pure-2obs+}
Let $H_1,H_2$ be arbitrary positive operators, let $E_1,E_2\in\R$,
let $s\subseteq\{1,2\}$, and let $f:\fS_{H_1,E_1,H_2,E_2}^s\to[-\infty,\infty]$
be a convex function on the intersection $\,\fS_{H_1,E_1,H_2,E_2}^s$ of sublevel 
and levels sets. If $f$ is either lower semicontinuous or upper semicontinuous 
and upper bounded, then
\begin{equation}\label{c-2-r-2obs+}
\sup\{f(\rho): \rho\in\fS_{H_1,E_1,H_2,E_2}^s\}
=\sup\{f(\rho): \rho\in\ext\fS_{H_1,E_1,H_2,E_2}^s\},
\end{equation}
where $\,\ext\fS_{H_1,E_1,H_2,E_2}^s$ is the set of pure states in 
$\fS_{H_1,E_1,H_2,E_2}^s$.
\par
If the domain of $f$ is the intersection $\fS_{H_1,E_1,H_2,E_2}$ of sublevel sets 
($s=\emptyset$), if $f$ is upper semicontinuous, and if one of the operators 
$H_1$ or $H_2$ has discrete spectrum of finite multiplicity, then the supremum on 
the right-hand side of \eqref{c-2-r-2obs+} is attained at a pure state in 
$\fS_{H_1,E_1,H_2,E_2}$.
\end{Thm}
\begin{proof}
By Corollary~\ref{cor:C++} and Theorem~\ref{thm:ExtSE+}, for any mixed state $\rho$ 
in $\fS_{H_1,E_1,H_2,E_2}^s$ there is a probability measure $\mu$ supported by pure 
states in the intersection $\fS_{H_1,E_1,H_2,E_2}$ of sublevel sets such that
\[
\rho=\int \sigma\mu(d\sigma)
\]
and such that $\mu(\ext\fS_{H_1,E_1,H_2,E_2}^s)=1$. The assumed properties of the 
function $f$ guarantee (see, f.i., \cite[the Appendix]{EM}) the validity of the 
Jensen inequality
\[
f(\rho)\leq \int f(\sigma)\mu(d\sigma),
\]
which implies the existence of a pure state $\sigma$ in $\fS_{H_1,E_1,H_2,E_2}^s$
that satisfies $f(\sigma)\geq f(\rho)$.
\par
If one of the operators $H_1$ or $H_2$ has discrete spectrum of finite multiplicity, 
then the set $\fS_{H_1,E_1,H_2,E_2}$ is compact. Hence, the set of extreme points 
$\,\ext\fS_{H_1,E_1,H_2,E_2}$ is compact by Theorem~\ref{thm:KMC}. This and the 
above arguments imply that the first supremum in \eqref{c-2-r-2obs+} is attained at 
a pure state in $\fS_{H_1,E_1,H_2,E_2}$ (provided that the function $f$ is upper semicontinuous).
\end{proof}
Of course, we may replace the convex function $f$ in 
Theorem~\ref{thm:sup-pure-2obs+} by the concave function $-f$ (and supremum by 
infimum). This idea is motivated by potential applications, since many important 
characteristics of a state in quantum information theory are concave lower 
semicontinuous and nonnegative. See the following examples.
\par
\begin{Exa}[The minimal output entropy of an energy-constrained quantum channel]
The \emph{von Neumann entropy} of a quantum state
$\rho$ in $\fS(\cH)$ is a basic characteristic of this state defined by the formula
$H(\rho)=\operatorname{Tr}\eta(\rho)$, where  $\eta(x)=-x\log x$ for $x>0$
and $\eta(0)=0$. The function $H(\rho)$ is  concave and lower semicontinuous on the
set~$\fS(\cH)$ and takes values in~$[0,+\infty]$, see for example \cite{H-SCI,L-2,W}.
\par
A quantum channel from a system $A$ to a system $B$ is a completely positive
trace-preserving linear map $\Phi:\fT(\cH)\to\fT(\cK)$ between the Banach
spaces $\fT(\cH)$ and $\fT(\cK)$, where $\cH$ and $\cK$ are Hilbert spaces
associated with the systems $A$ and $B$, respectively. In the analysis of
information abilities of quantum channels, the notion of the minimal output
entropy of a channel is widely used \cite{H-SCI,GP&Co,MGH,Man,Shor}. It is defined
as
\begin{equation}\label{eq:min-output}
H_{\rm min}(\Phi)=\inf_{\rho\in\fS(\cH)} H(\Phi(\rho))=\inf_{\varphi\in\cH_1} H(\Phi(\ket{\varphi}\!\!\bra{\varphi})),
\end{equation}
where $\cH_1$ is the unit sphere in $\cH$, and $\ket{\varphi}\!\!\bra{\varphi}$
denotes the projector of rank one onto the line spanned by $\varphi\in\cH_1$.
The second equality of \eqref{eq:min-output} follows
from the concavity of the function $\rho\mapsto H(\Phi(\rho))$ and from the
possibility to decompose any mixed state into a convex combination of pure states.
\par
In studies of infinite-dimensional quantum channels, it is reasonable to impose the
energy-constraint on input states of these channels. So, alongside with the minimal
output entropy $H_{\rm min}(\Phi)$, it is reasonable to consider its constrained
versions (cf.~\cite{Man})
\begin{eqnarray}
\label{CME-1}
H_{\rm min}(\Phi,H,E) =
\inf_{\rho\in\fS(\cH):\Tr H\rho\leq E} H(\Phi(\rho)),\\[.5\baselineskip]
\label{CME-2}
H^{=}_{\rm min}(\Phi,H,E) =
\inf_{\rho\in\fS(\cH): \Tr H\rho=E} H(\Phi(\rho)),
\end{eqnarray}
where $H$ is a positive operator, the {\em energy observable}.
In contrast to the unconstrained case, it is not obvious that the infima in
(\ref{CME-1}) and (\ref{CME-2}) can be taken only over pure states satisfying the
conditions $\Tr H\rho\leq E$ and $\Tr H\rho=E$ correspondingly. In \cite{Man} it
is shown that this holds in the finite-dimensional settings. The above
Theorem~\ref{thm:sup-pure-2obs+} allows to prove the same assertion for an
arbitrary infinite-dimensional channel $\Phi$ and any energy observable $H$.
\end{Exa}
\begin{Cor}\label{cor:minimum-output}
Let $H$ be an arbitrary positive operator and let
$E$ be greater than the infimum of the spectrum of $H$. Then both infima in (\ref{CME-1}) and (\ref{CME-2}) can be taken over pure states, i.e.
\begin{eqnarray}
\label{CME+1}
H_{\rm min}(\Phi,H,E) =
\inf_{\varphi\in\cH_1:\,\braket{\varphi|H|\varphi}\leq E} 
H(\Phi(\ket{\varphi}\!\!\bra{\varphi})),\\[.5\baselineskip]
\label{CME+2}
H^{=}_{\rm min}(\Phi,H,E) =
\inf_{\varphi\in\cH_1:\,\braket{\varphi|H|\varphi}=E} 
H(\Phi(\ket{\varphi}\!\!\bra{\varphi})).
\end{eqnarray}
If the operator $H$ has discrete spectrum of finite multiplicity, then the
infimum in (\ref{CME+1}) is attained at a unit vector.
\end{Cor}
\begin{proof}
By Theorem~\ref{thm:sup-pure-2obs+}, it suffices to note that the function 
$\rho\mapsto H(\Phi(\rho))$ is concave nonnegative and lower semicontinuous (as 
a composition of a continuous and a lower semicontinuous function).
\end{proof}
Corollary~\ref{cor:minimum-output} simplifies the definitions of the
quantities $H_{\rm min}(\Phi,H,E)$ and $H^{=}_{\rm min}(\Phi,H,E)$ significantly. It
also shows that
$$
H_{\rm min}(\widehat{\Phi},H,E)=H_{\rm min}(\Phi,H,E)\quad \text{and}
\quad H^{=}_{\rm min}(\widehat{\Phi},H,E)=H^{=}_{\rm min}(\Phi,H,E),
$$
where $\widehat{\Phi}$ is a complementary channel to the channel $\Phi$, since
for any pure state $\rho$ we have $H(\widehat{\Phi}(\rho))=H(\Phi(\rho))$, see
Section~8.3 of \cite{H-SCI}.
\par
\begin{Exa}[On the definition of the operator E-norms]
On the algebra $\fB(\cH)$ of all bounded operators one can consider the family
$\{\|A\|_E^H\}_{E>0}$ of norms induced by a positive operator $H$ with the infimum
of the spectrum equal to zero \cite{ECN}. For any $E>0$ the norm $\|A\|_E^H$ is
defined as
\begin{equation}\label{ec-on}
 \|A\|^{H}_E\doteq \sup_{\rho\in\fS(\cH):\Tr H\rho\leq E}\sqrt{\Tr A\rho A^*}.
\end{equation}
These norms, called operator \emph{E}-norms, appear as ``doppelganger'' of
the energy-constrained Bures distance between completely positive linear maps in
the generalized version of the Kretsch\-mann-Schlingemann-Werner theorem
\cite[Section 4]{ECN}.
\par
For any $A\in\fB(\cH)$ the function $E\mapsto\|A\|_E^H$ is concave and tends to
$\|A\|$ (the operator norm of $A$) as $E\to+\infty$. All the norms $\|A\|_E^H$
are equivalent (for different $E$ and fixed $H$) on $\fB(\cH)$ and generate a
topology depending on the operator $H$. If $H$ is an unbounded operator then
this topology is weaker than the norm topology on $\fB(\cH)$, it coincides with
the strong operator topology on bounded subsets of $\fB(\cH)$ provided that the
operator $H$ has discrete spectrum of finite multiplicity.
\par
If we assume that the supremum in (\ref{ec-on}) can be taken only over pure states
$\rho$ such that $\Tr H\rho\leq E$ then we obtain the following simpler definition
\begin{equation}\label{ec-on-b}
 \|A\|^H_E\doteq \sup_{\varphi\in\cH_1,\braket{\varphi|H|\varphi}\leq E}\|A\varphi\|,
\end{equation}
which shows the sense of the norm $\|A\|^H_E$ as a constrained version of the operator
norm $\|A\|$. In \cite{ECN} the above assumption was proved only in the case when the
operator $H$ has discrete spectrum of finite multiplicity. 
Theorem~\ref{thm:sup-pure-2obs+} (applied to the continuous affine function
$f(\rho)=\Tr A\rho A^*$) allows to fill this gap.
\end{Exa}
\begin{Cor}
For an arbitrary positive operator $H$, the definitions (\ref{ec-on}) and 
(\ref{ec-on-b}) coincide for any $A\in\fB(\cH)$.
\end{Cor}
%
%%%%%%%%%%%%%%%%%%%%%%%%%%%%%%%%%%%%%%%%%%%%%%%%%%%%%%%%%%%%%%%%%%%%%%%%%%%%
%%%%%%%%%%%%%%%%%%%%%%%%%%%%%%%%%%%%%%%%%%%%%%%%%%%%%%%%%%%%%%%%%%%%%%%%%%%%
%%%%%%%%%%%%%%%%%%%%%%%%%%%%%%%%%%%%%%%%%%%%%%%%%%%%%%%%%%%%%%%%%%%%%%%%%%%%
%%%%%%%%%%%%%%%%%%%%%%%%%%%%%%%%%%%%%%%%%%%%%%%%%%%%%%%%%%%%%%%%%%%%%%%%%%%%
%%%%%%%%%%%%%%%%%%%%%%%%%%%%%%%%%%%%%%%%%%%%%%%%%%%%%%%%%%%%%%%%%%%%%%%%%%%%
%
\vspace{\baselineskip}
\noindent
{\footnotesize
Acknowledgements.
The first author thanks M.\,R.~Galarza and M.\,M.\ and J.~Weis for hosting
him while working on this project. The second author is grateful to A.\,S.~Holevo and
G.\,G.~Amosov for useful discussions.
Both authors thank F.~de Melo for the idea to study constraints under several
observables.}
%
%%%%%%%%%%%%%%%%%%%%%%%%%%%%%%%%%%%%%%%%%%%%%%%%%%%%%%%%%%%%%%%%%%%%%%%%%%%%
%%%%%%%%%%%%%%%%%%%%%%%%%%%%%%%%%%%%%%%%%%%%%%%%%%%%%%%%%%%%%%%%%%%%%%%%%%%%
%%%%%%%%%%%%%%%%%%%%%%%%%%%%%%%%%%%%%%%%%%%%%%%%%%%%%%%%%%%%%%%%%%%%%%%%%%%%
%%%%%%%%%%%%%%%%%%%%%%%%%%%%%%%%%%%%%%%%%%%%%%%%%%%%%%%%%%%%%%%%%%%%%%%%%%%%
%%%%%%%%%%%%%%%%%%%%%%%%%%%%%%%%%%%%%%%%%%%%%%%%%%%%%%%%%%%%%%%%%%%%%%%%%%%%
%
%
\bibliographystyle{plain}

%
%
%%%%%%%%%%%%%%%%%%%%%%%%%%%%%%%%%%%%%%%%%%%%%%%%%%%%%%%%%%%%%%%%%%%%%%%%%%%%
%%%%%%%%%%%%%%%%%%%%%%%%%%%%%%%%%%%%%%%%%%%%%%%%%%%%%%%%%%%%%%%%%%%%%%%%%%%%
%%%%%%%%%%%%%%%%%%%%%%%%%%%%%%%%%%%%%%%%%%%%%%%%%%%%%%%%%%%%%%%%%%%%%%%%%%%%
%%%%%%%%%%%%%%%%%%%%%%%%%%%%%%%%%%%%%%%%%%%%%%%%%%%%%%%%%%%%%%%%%%%%%%%%%%%%
%%%%%%%%%%%%%%%%%%%%%%%%%%%%%%%%%%%%%%%%%%%%%%%%%%%%%%%%%%%%%%%%%%%%%%%%%%%%
%
\vspace{\baselineskip}
\parbox{10cm}{%
Stephan Weis\\
Theisenort 6\\
96231 Bad Staffelstein\\
Germany\\
e-mail \texttt{maths@weis-stephan.de}}
\vspace{\baselineskip}
\par\noindent
\parbox{10cm}{%
Maksim Shirokov\\
Steklov Mathematical Institute\\
Moscow\\
Russia\\
e-mail \texttt{msh@mi.ras.ru}}
\end{document}